\documentclass[11pt]{article}

\usepackage[margin=1in]{geometry}  
\usepackage{graphicx}              
\usepackage{amsmath}               
\usepackage{amsthm}
\usepackage{amsfonts}              
\usepackage{amssymb}
\usepackage{caption}
\usepackage{subfigure}
\usepackage[nottoc,numbib]{tocbibind}
\usepackage{verbatim}
\usepackage{bbm}
\usepackage{setspace}
\usepackage{url}
\usepackage{hyperref}

\newcommand{\NN}{\mathbb{N}}      
\newcommand{\PP}{\mathbb{P}}
\newcommand{\EE}{\mathbb{E}}
\newcommand{\HH}{\mathbb{H}}
\newcommand{\CR}{\mathrm{CR}}
\newcommand{\II}{\mathbbm{1}}

\newcommand{\abs}[1]{\lvert#1\rvert}

\makeatletter \renewenvironment{proof}[1][\proofname] {\par\pushQED{\qed}\normalfont\topsep6\p@\@plus6\p@\relax\trivlist\item[\hskip\labelsep\bfseries#1\@addpunct{.}]\ignorespaces}{\popQED\endtrivlist\@endpefalse} \makeatother

\newtheorem{theorem}{Theorem}[section]
\newtheorem*{theorem*}{Theorem}
\newtheorem{lemma}[theorem]{Lemma}
\newtheorem{proposition}[theorem]{Proposition}
\newtheorem*{proposition*}{Proposition}
\newtheorem{corollary}[theorem]{Corollary}
\newtheorem*{corollary*}{Corollary}

\newtheorem{claim}[theorem]{Claim}
\newtheorem{question}[theorem]{Question}

\theoremstyle{remark}
\newtheorem*{remark}{Remark}

\theoremstyle{definition}
\newtheorem*{definition}{Definition}

\begin{document}

\author{Juhan Aru\thanks{\newline Affiliation: ENS Lyon (UMPA)\newline
Email: juhan.aru@cantab.net\newline Partially supported by the ANR grant MAC2 10-BLAN-0123.}}
\title{KPZ relation does not hold for the level lines and SLE$_\kappa$ flow lines of the Gaussian free field}
\date{}
\maketitle
\abstract{In this paper we mingle the Gaussian free field, the Schramm-Loewner evolution and the KPZ relation in a natural way, shedding new light on all of them. Our principal result shows that the level lines and the SLE$_\kappa$ flow lines of the Gaussian free field do not satisfy the usual KPZ relation. In order to prove this, we have to make a technical detour: by a careful study of a certain diffusion process, we provide exact estimates of the exponential moments of winding of chordal SLE curves conditioned to pass nearby a fixed point. This extends previous results on winding of SLE curves by Schramm.
}

\date{}
\maketitle

{
\small
\tableofcontents
}

\newpage

\section{Introduction}

This paper combines in a way three beautiful mathematical concepts, all having three-letter abbreviations: the Gaussian free field (GFF), the Schramm-Loewner evolution (SLE) and the KPZ relation. 

The background motivation comes from statistical physics. Statistical physics models on Euclidean lattices are often difficult to study. Even when for the self-avoiding walk on the hexagonal lattice we know the connective constant \cite{HUGO}, we are for example only beginning to gather any rigorous results at all on the square lattices. Also, we still hope for proofs of critical percolation exponents on the same lattice. 

However, in the eighties three physicists Knizhnik, Polyakov and Zamolodchikov \cite{KPZ} came up with a far-reaching strategy for studying these models. The proposed plan was to study them in a random environment, or in what they called the Quantum Gravity regime, and then translate the results back to the Euclidean setting. This was a fruitful idea as the study of many models becomes easier in these random environments, and even more - the so called KPZ relation gives an exact translation for critical exponents back to the Euclidean case \cite{DUP1,DUP2,CUR}.

Mathematically, however, the understanding of the KPZ relation is still scarce. Mainly, the problem is that in higher than one dimension, we do not yet have a suitable continuum model for the random environment that would allow understanding of the KPZ relation. Even though random planar maps have been shown to converge to a candidate random metric space \cite{MIER,LG}, we are still missing a conformal structure on these spaces, thus making it hard to relate models on these spaces with our usual models on Euclidean lattices.

Still, recently there has been progress in understanding the KPZ relation. In one dimension, we have a quite good understanding \cite{BS}. For two dimensions, a more mundane version of the random environment has helped us. Namely, whereas ideally we would like to establish the KPZ relation in a random metric space with a certain topology, we can already give meaning to the KPZ relation when we model the random environment by a random measure on a two-dimensional domain. This measure is called the Liouville measure \cite{DS,GAR}.

In this context of the Liouville measure the KPZ relation can be shown to rigorously relate Euclidean and Quantum fractal dimensions \cite{DS,RV}. There is, however, a little catch - all the proofs only work for deterministic sets and sets independent of the random environment. However, in at least a few cases the statistical physics models are coupled with the random environment, as for example in the Ising model. Though expected, it is not a priori clear whether our sets of interest, as for example the interface boundaries, will become independent in the continuum limit. Hence it is also interesting to ask to which extent the KPZ relation holds for sets depending on the measure.

In this article, we treat the case of most natural sets coupled with the Liouville measure - the SLE$_\kappa$ curves corresponding to interface boundaries in statistical physics models. One way of coupling the SLE lines with the GFF and the Liouville measure is using a conformal welding of two Liouville quantum surfaces \cite{ZIPPER,DSSLE}. This ought to correspond to gluing random planar maps in the discrete setting. We already know that in this case one recovers a KPZ relation, if instead of volume measures one considers boundary measures on the SLE \cite{ZIPPER,DSSLE}. In what follows, we show that on the other hand the usual KPZ relation does not hold for the SLE$_\kappa$ with $0 < \kappa < 8$ coupled with the GFF as level lines ($\kappa = 4$) or flow lines of the field. Notice that this implies that the KPZ relation is of very different character than the Kaufman's theorem on dimension doubling of the Brownian motion. It can also been seen as evidence that, indeed, in the continuum limit the interface boundaries have to become independent of the random environment. 

On the way towards the final proof, we have to find new precise estimates of the exponential moments of winding of chordal SLE curves around points conditioned to be close to the curve. This goes beyond Schramm's analysis in his seminal paper introducing the SLE curves \cite{SCHRAMM} and could be of independent interest. 
\begin{figure}[h]

\captionsetup{width=0.8\textwidth}
\centering
\begin{minipage}[b]{0.4\textwidth}
\centering\includegraphics[width=\textwidth]{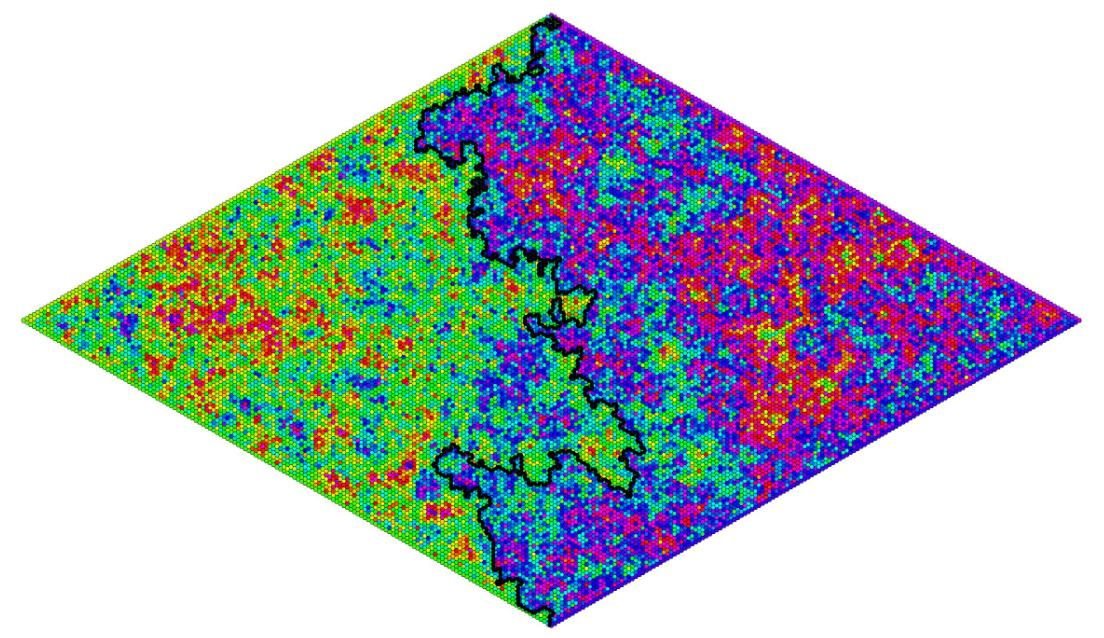}
\end{minipage}
\hspace{0.5cm}
\begin{minipage}[b]{0.4\textwidth}
\centering
\includegraphics[width=\textwidth]{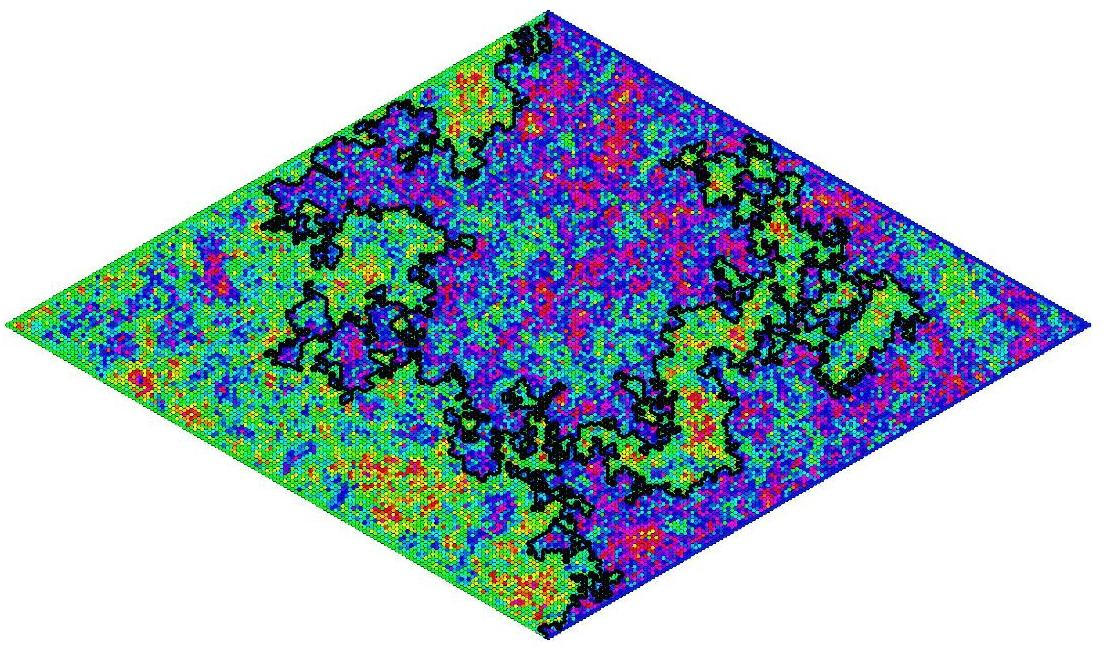}
\end{minipage}
\caption{{\small On the left, we see the flow line coupling of the SLE$_{8/3}$ and on the right the zero level line coupling. The colours indicate the height of the GFF. Notice that whereas the zero level line - as by definition it should - really moves along the boundary of positive and negative heights, SLE$_{8/3}$ also keeps close to this boundary.}}
\label{fig1}
\end{figure}

\paragraph{Outline and main results}\mbox{}\\
We start the paper by giving a concise description of the key constructions of the paper: the GFF, the SLE, the Liouville measure. Then we discuss at more length the different versions of the KPZ relation in the literature \cite{DS} \cite{RV} and propose yet a third one. Indeed, whereas we set off to prove our claim for the almost sure Hausdorff version, an intermediate step of determining what we call the expected quantum Minkowski dimension of these lines became useful. We also discuss how one can come up with easier, but less natural counterexamples for the KPZ relation.

Next, in section 3 we study the introduced notion of the expected Minkowski dimension. We prove the relevant KPZ relation and show that, as expected, the expected quantum Minkowski dimension is always larger than the quantum Hausdorff dimension introduced in \cite{RV}.

After these preliminaries, we are ready to attack the zero level lines in section 4. There is a simple proof, a matter of only putting our intuition on a rigorous grounding: the fact that the GFF is forced to be low near the zero level line, means that the Liouville measure is small, hence it is easier to cover the zero level line and both the expected quantum Minkowski and quantum Hausdorff dimensions are smaller than predicted by the KPZ relation.

Handling SLE$_\kappa$ flow lines is considerably harder and needs some technical work on the SLE curves. In section 5 we derive up to multiplicative constants the exponential moments of the winding of the chordal SLE curves, conditioned to arrive close at points. More explicitly, we prove the following theorem:

\begin{theorem}

Let $\CR_0$ be the conformal radius of a fixed point $z_0$ in the upper half plane. Fix $0 <  \kappa < 8$ and let $\tau$ be the time that SLE first cuts $z_0$ from infinity. Denote by $H_\tau$ the SLE slit domain component containing $z_0$. Then, for $\epsilon > 0$ sufficiently small, conditioned on $\CR(z_0,H_\tau) \in [\epsilon, C\epsilon]$ with $C > 1$, the exponential moments of the winding $w(z_0)$ around the point $z_0$ are given by $e^{\lambda w(z_0)} \asymp \epsilon^{-\lambda^2\kappa/8}$, where the implied constants depend on $\kappa, \lambda$ and for fixed $\kappa$ can be chosen uniform for $\abs{\lambda} < \lambda_0$ for any choice of $\lambda_0 > 0$.

\end{theorem}

We do this by using a diffusion process related to SLE already in previous papers \cite{BEF,LSW,WSS}. We need, however, to study this process in finer detail, and provide good control of the eigenvalues and eigenfunctions of the respective generator. The whole section is a bit technical, but both the result and methods could be of independent interest.

Thereafter, in section 6 we find the expected quantum Minkowski dimension of the SLE$_\kappa$ flow lines by introducing a non-standard Whitney decomposition that is based on the conformal radius instead of the Euclidean radius. This allows us to work off the curve, where things get singular, and to use the results on winding obtained.  The final result, containing also the previous work on zero level lines, can be stated as follows:

\begin{theorem}
Consider the Liouville measure with $\gamma < 2$ in the unit disc and let $0 < \kappa < 8$. Then the expected quantum Minkowski dimension of the SLE$_\kappa$ flow lines is given by $q_{M,E} < 1$ satisfying
\[d_M = (2+\gamma^2/2)q_{M,E} - \gamma^2(1-\kappa/4)^2q_{M,E}^2/2\]
where $d_M$ is the Minkowski dimension of the respective SLE curve. 
\end{theorem}

As the usual KPZ relation is given by $d_M = (2+\gamma^2/2)q_{M,E} - \gamma^2q_{M,E}^2/2$, this in particular means that for $0 < \kappa < 8$ the KPZ relation is not satisfied for the expected Minkowski dimension of the $SLE_\kappa$ in forward coupling with the GFF. Notice that in the limits $\kappa \downarrow 0, \kappa \uparrow 8$ we regain the KPZ relation. Using the fact that the quantum Hausdorff dimension is dominated by the expected quantum Minkowski dimension, we also deduce the following corollary 

\begin{corollary}
Consider the Liouville measure with $\gamma < 2$ in the unit disc and let $0 < \kappa < 8$. Then almost surely the quantum Hausdorff dimension for the flow lines SLE$_\kappa$ is below the dimension predicted by KPZ relation and hence the KPZ relation is not satisfied in the almost sure Hausdorff version. 
\end{corollary}

This incompatibility with the usual KPZ relation is illustrated by the following figure, where we fixed $\gamma = \sqrt{2}$, the dotted line represents the usual KPZ relation, and the solid line the actual quadratic relation satisfied by the expected quantum Minkowski dimension.
\begin{center}
    \includegraphics[width = 0.5\textwidth, keepaspectratio]{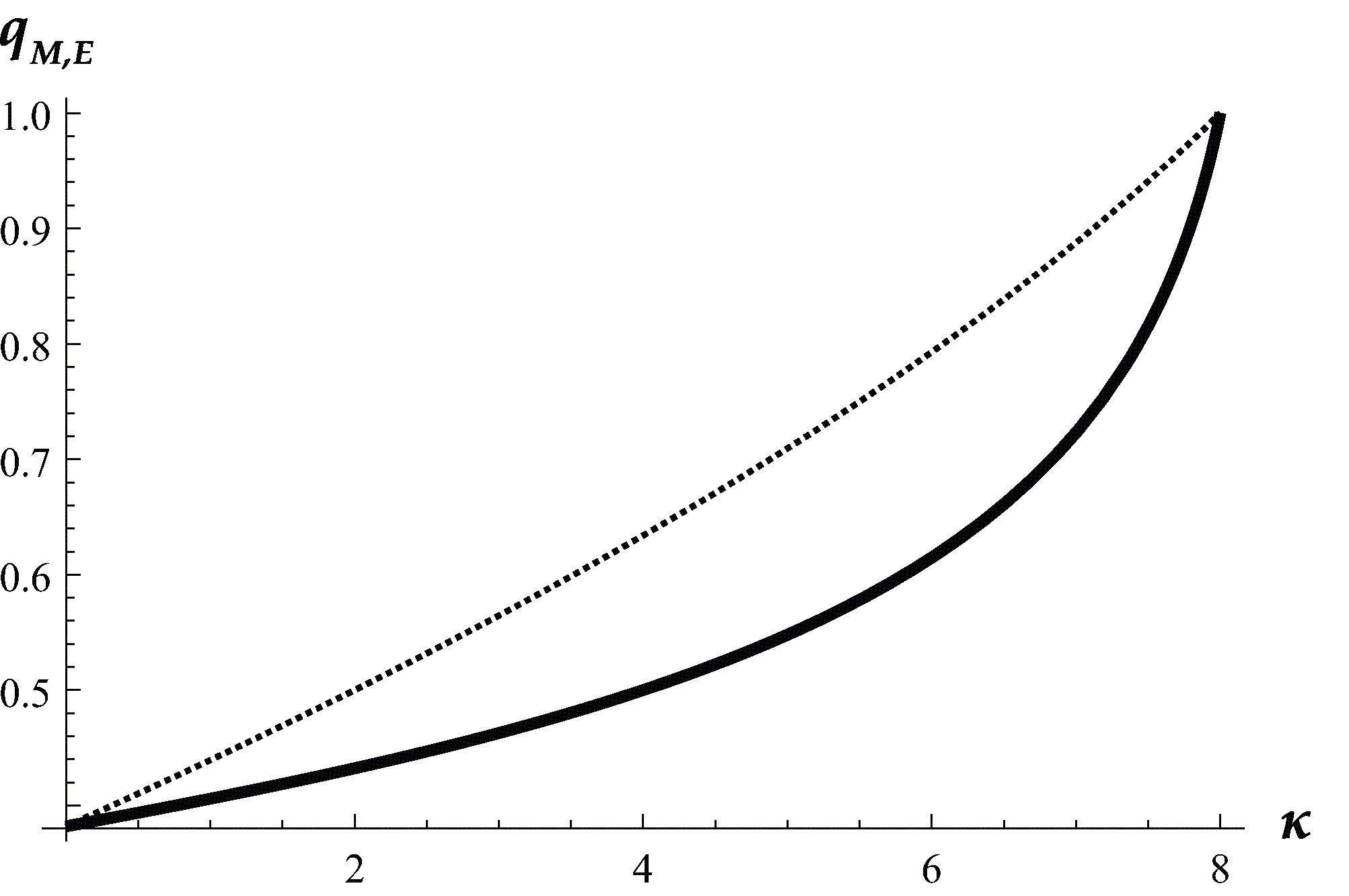}
    \end{center}

After having proved the main theorems of the paper, we return to a less rigorous level and finish the article with a section on some speculations and open questions.

\paragraph{Notations}\mbox{}\\
Multiplicative constants arriving in our calculations are of little importance and hence we use the big-$O$ and the $\lesssim$ notation. We write $f \lesssim g$ to mean that $f(x) \leq Mg(x)$ for all $x$ in the range of definition and some absolute constant $M$. If we need to make explicit the dependence of these constants, we write $\lesssim_{c,d}$ etc. Similarly, we use the notation $g \asymp f$, in cases where simultaneously $g \lesssim f$ and $f \lesssim g$. We hope this makes the calculations more readable.

\paragraph{Acknowledgements}\mbox{}\\
I would like to thank Christophe Garban for his direction towards this problem, his generous advice, encouragement and helpful comments on all possible versions of the manuscript. I am thankful to Wendelin Werner for introducing me to the playfield of the SLE and the GFF. I also thank Bruno Sevannec for helping to look for references in the literature of PDEs, Scott Sheffield for kindly sharing the images of the SLE and GFF coupling; Kaie Kubjas, Moses Tannenbaum and Michele Triestino for comments on the draft. I also acknowledge the partial support of the ANR grant MAC2 10-BLAN-0123.

\section{Preliminaries}

We will next introduce shortly the mathematical setting of our problem and discuss in a bit more length different possible formulations of the KPZ relation.

\subsection{GFF \& SLE \& their couplings}
\paragraph{GFF}\mbox{}\\
The Gaussian free field (GFF) is a model for random surfaces (though it is not mathematically a surface itself) and represents the Euclidean bosonic massless field in quantum field theories \cite{GFF}. It can also be conceptualised as a Brownian Motion with 2 time dimensions. Finally, due to its roughness, it is not defined point-wise and hence the underlying mathematical setting is that of random distributions.

We will give a possible definition of the zero boundary GFF in the upper half plane, but for a thorough treatment refer to \cite{GFF,DUB} and for a shorter introduction to \cite{GAR}. To define the GFF, denote by $\mathcal{H}_0(\HH)$ the set of smooth functions compactly supported inside $\HH$. Let $\mathcal{H}(\HH)$ be its closure with respect to the Dirichlet inner product and $\mathcal{H}(\HH)^{-1}$ the Hilbert space dual.

\begin{definition} [Zero boundary Gaussian free field]
The zero boundary GFF $h$ on the upper half plane $\HH$ can be defined as the zero-mean Gaussian process on the space $\mathcal{H}(\HH)^{-1}$ with the covariance kernel given by the Green's function $G_{\HH}(x,y)$ with zero boundary condition. In other words, to each $f \in \mathcal{H}(\HH)^{-1}$ the GFF attaches a Gaussian $h(f)$ such that $\EE(h(f)) = 0$ and $\EE(h(f)h(g)) = \frac{1}{2\pi}\int_{\HH \times \HH} f(x)g(y)G_{\HH}(x,y)dxdy$.
\end{definition}

The Gaussian free field is conformally invariant, if seen as a "height function" \cite{GFF}. Hence in fact this definition provides one for any simply-connected connected domain with Jordan boundary of the plane. In particular we will also use the GFF in the unit disc. Also, here we also work with the GFF with fixed boundary conditions - this just means we have a zero-boundary GFF plus an harmonic extension of the fixed boundary conditions. In all cases of this paper the harmonic extension can be seen to exist.

\paragraph{SLE}\mbox{}\\
Schramm-Loewner evolution (SLE) is a family of random curves, that were invented to describe the interfaces of models in statistical physics \cite{SCHRAMM}. They are conformally invariant processes that satisfy the domain Markov property. For a thorough introduction we refer to either \cite{WERNER} \cite{LAWLER}. 

Whereas one can talk of chordal, radial and whole-plane SLE-s, we here concentrate only on the chordal version. We define it in the upper half-plane $\HH$, but due to conformal invariance this of course gives the definition for any simply connected domain. SLE-s are defined via a family of conformal maps, and these conformal maps themselves are defined using the so called Loewner differential equation.

\begin{definition}[Loewner differential equation]
Let $\zeta(t)$ be a continuous real-valued function. Then for any $z \in \HH$ define $g_0(z) = z$ and
\[\partial_tg_t(z) = \frac{2}{g_t(z)-\zeta(t)}\]
defined up to $\tau(z) = \sup_{t\geq0}\min\abs{g_t(z)-\zeta(t)} > 0$.
\end{definition} 

If we write $K_t = \{z: \tau(z) \leq t \}$ then the equation above defines a family of conformal maps from the decreasing domains $H_t = \HH\backslash K_t$ back to the upper half plane. The family $K_t$ is called the Loewner chain. The randomness part enters by defining the driving function $\zeta(t)$ to be a multiple of the standard Brownian motion.

\begin{definition}[Chordal SLE]
Let $B_t$ denote a standard Brownian motion. Then the Loewner chain given by the driving function $\zeta(t) = \kappa B_t$ with $\kappa \geq 0$ is called an SLE$_\kappa$. 
\end{definition}

In this paper, we want to normalize the map such that the tip of the curve maps to zero. This cane be done by just setting 
\[f_t(z) = g_t(z) - \zeta(t)\]

It has been shown that the SLE chains are almost surely generated by a curve \cite{ROHDESCHRAMM}. We will be interested in the quantum fractal dimension of these curves, when coupled with the GFF. We use the known fact that the Hausdorff dimension of SLE$_\kappa$ curve for $\kappa \leq 8$ is $1+\kappa/8$, first proved in its entirety \cite{BEF}.

\paragraph{GFF \& SLE couplings}\mbox{}\\
The Gaussian free field and Schramm-Loewner evolutions are coupled in two beautiful ways \cite{DUB,ZIPPER}. One way is to see SLE curves as interfaces for glueing together two Liouville measures \cite{ZIPPER}. In this paper we work with the other way, which gives SLE curves a geometric meaning, when interpreting GFF as a random surface \cite{DUB,ZIPPER}. Set 
\[\lambda = \frac{\pi}{\sqrt{\kappa}}\]

Firstly, $SLE_4$ can be seen as, or maybe rather forced to be the zero level line of the GFF \cite{SS}:

\begin{theorem}[Zero level lines of the GFF]
\label{thmCL}
Let $\eta$ be a chordal SLE$_4$ curve in $\HH$ and $h$ the GFF in $\HH$ with boundary conditions $-\lambda$, $\lambda$ on the negative and positive real axis respectively. Then there is a coupling $(h, \eta)$ such that 
\begin{itemize}
\item the marginal of $h$ can be obtained by sampling the SLE up to some stopping time $T$, and then sampling an independent GFF in the slit domain with boundary conditions set to $-\lambda$ on negative real axis and to the left of the SLE, and $\lambda$ to the right of the SLE and on the positive real axis
\item $\eta$ is a measurable with respect to $h$
\end{itemize}

\end{theorem}

\begin{remark}
It is shown in \cite{SS} that the GFF of the slit domain is well-defined as a distribution on the whole of $\HH$. This will be important for us. Also notice that the harmonic correction term introduced by boundary conditions is not well-defined on SLE. However, it can be set to zero - see the discussion following the statement on theorem 1.1. in \cite{ZIPPER}. Finally, notice that this harmonic correction coming from the boundary conditions is uniformly bounded.
\end{remark}

Secondly, $SLE_\kappa$ for $\kappa > 0$ can be seen as flow lines of the GFF \cite{ZIPPER,DUB,IM}. Whereas intuitively level lines are clear, flow lines are a bit harder to interpret. Nice pictures with nice explanations can be found in \cite{IM}, but also figures 1 and 2 of the paper should provide some insight. In short and without rigour, for flow lines at any point, the angular derivative is given by a multiple of the field height. 

\begin{theorem}[Flow lines of the GFF]
\label{thmC}
For $0 \leq \kappa < 4$, let $\eta_\kappa$ be a chordal SLE$_\kappa$ curve in $\HH$ and $h$ the GFF in $\HH$ with boundary conditions $-\lambda$, $\lambda$ on the negative and positive real axis respectively. Then there is a coupling $(h, \eta_\kappa)$ such that 
\begin{itemize}
\item the marginal of $h$ can be obtained by 
\begin{itemize}
\item sampling the SLE $\eta_\kappa$ up to some time $T$
\item then sampling an independent GFF in the slit domain with boundary conditions as above: $-\lambda$ on negative real axis and to the left of the SLE, and $\lambda$ to the right of the SLE and on the positive real axis \item and finally, subtracting $\chi \arg f'_T$ where $\chi = 2/\sqrt{\kappa} - \sqrt{\kappa}/2$ and $f_T$ is the normalized SLE map
\end{itemize}
\item $\eta$ is a measurable with respect to $h$
\end{itemize}
\end{theorem}

Again, there are some remarks to be made. First, this coupling reduces to the level line coupling for $\kappa = 4$ as then $\chi = 0$. As above, GFF in the slit domain can be extended to the whole plane, and the harmonic correction term (this time possibly unbounded!) can be still set to zero on the curve. Second, $\arg f'_T(z) = \operatorname{Im} \log f'_T(z)$ measures the winding of the SLE curve with respect to the point $z$. We require the argument to be continuous in the slit domain and tend to $0$ at infinity. The winding is discussed in section 5 of the paper, but we also refer to \cite{IM}. 

Also, in fact $\kappa < 4$ is no real restriction, everything can also be stated for $\kappa > 4$. One needs to just take extra care as the SLE curve is no longer simple: first, the winding for any point needs to be calculated just before the point gets separated from infinity by the curve i.e. as a limit $\lim_{t\uparrow T'} f_t(z)$ with $T' = T \wedge \tau(z)$, where $\tau(z)$ is the first time $z$ is separated from the infinity by the SLE curve. Second, a separate and independent GFF needs to be defined in each isolated domain (they all extend similarly to the whole of $\HH$) and for the boundary conditions one needs to take into account in which direction the loops were closed. For details, see \cite{ZIPPER} or \cite{IM}. 

Finally, we remark that more generally SLE$_{\kappa,\underline{\rho}}$ processes satisfying certain conditions give flow lines of the GFF, see \cite{IM} for details.

\subsection{Liouville measure}

Liouville measure should be the right model for a random measure underlying the study of statistical physics models in their "quantum gravity" form. It is one step short of the actual aim - the Liouville metric.

Mathematically, the Liouville measure ought to be the exponential of the Gaussian free field. However, as GFF is formally a distribution, one needs to define the Liouville measure using some kind of regularization process. There are many ways of achieving this, the roots going back to the beautiful work of Kahane \cite{KAHANE} on Gaussian Multiplicative Chaos. Different ways of defining the Liouville measure and their equivalence are discussed in greater detail in \cite{GMC}.

In this article we use the circle-averaging regularization as used in Duplantier \& Sheffield \cite{DS}. As explained below, this suits our needs better.

\paragraph{Liouville measure}\mbox{}\\
In \cite{DS} the following process is used to define the Liouville measure in any sufficiently nice domain $D$:

\begin{itemize}
\item First, regularize the field by taking circle averages around each point, i.e. set 
\[h_\delta(z) = h(\rho^z_{\delta_n})\]
where by $\rho^z_{\delta_n}(z) \in \HH^{-1}$ we denote the distribution giving unit mass to the circle of radius $\delta_n$ around the point $z$.
\item Now let $0 < \gamma < 2$ and define the $\delta-$approximate Liouville measures as
\[d\mu_\delta(z) = \delta^{\gamma^2/2}e^{\gamma h_\delta(z)}dz\]
\end{itemize}

\begin{remark}
The regularized GFF corresponds to a Gaussian field with the covariance kernel given by
\[G_{\delta}(x,y) = \log \frac{1}{\delta \vee \abs{x-y}} + \tilde{G}_{\delta}(x,y)\]
Here $\tilde{G}_{\delta}(x,y)$ is the harmonic extension to the domain of the the function $-\log \frac{1}{\delta \vee \abs{x-y}}$ on boundary. See \cite{DS} for details.
\end{remark}

Then the following theorem can be then taken as definition of the Liouville measure \cite{DS}: 

\begin{theorem}\label{Liouville}
Let $D$ be a domain. For $0 \leq \gamma < 2$, along powers of two in the interior of $D$, then almost surely $\delta-$approximate Liouville measures weakly converge to a non-degenerate random measure $\mu_\gamma$, called the Liouville measure.  This measure is measurable w.r.t zero-boundary GFF $h$.
\end{theorem}

\begin{remark}
Often we denote $\mu = e^{\gamma h}$, as $\gamma$ can be taken to be a fixed parameter $0 < \gamma < 2$ throughout the paper. 
\end{remark}

The renormalization term $\delta^{\gamma^2/2}$ is there to compensate for the growing variance and in this case uniform over the domain. As a result the field is lower near the boundary as the Green's function giving the covariance structure is lower near the boundary. This is illustrated for example by the fact that $\EE \mu_\delta(z) = \CR(z,D)^{\gamma^2/2}$, where $\CR(z,D)$ denotes the conformal radius at the point $z$ (which is though rigorously at distance at least $\delta$ from the boundary).

In the Gaussian Multiplicative Chaos approach one usually renormalizes by setting 
\[dM_\delta(z) = e^{\gamma h_\delta(z)-\gamma^2/2\EE(h_\delta(z))}dz\]
i.e. directly by the variance. This setting is more comfortable for a much wider selection of regularization procedures, as you do not need to know explicitly the variance on each regularization step \cite{RRV}. For us, this type renormalization, however, poses a little problem.

Namely, we we want to couple the Liouville measure with the SLE in a similar way to the coupling of the GFF and the SLE. Recall that in order to sample a GFF in this coupling, we start by first sampling an SLE, then choosing an independent GFF in the slit domain and adding some harmonic correction terms. For the Liouville measure we would hence like to also start from sampling the SLE, then defining our Gaussian field in the slit domain and finally construct the Liouville measure such that its distribution is that of the whole domain. 

Now, for the renormalization in the definition of the Liouville measure above, the regularized field $h_\delta$ corresponds to taking circle averages around each point, or in other words evaluating the distribution $h$ at unit measures on $delta-$circles. It does not matter whether we constructed the underlying by sampling first the SLE, then the GFF with harmonic correction terms or otherwise. On the other hand, with the variance-based renormalization, under the conditioned measure the renormalization term does not change and in some sense the zero level line is just not seen.

Thus in this article we use the Liouville measure regularized as in the definition above. However, we shall however find it useful to change the renormalization for some calculations, in order to use some machinery developed in that context \cite{KAHANE}\cite{RRV}.

\subsubsection{The 2D KPZ relations for the Liouville measure}

We now introduce two canonical versions for the KPZ relation in 2D quantum gravity, and propose yet another one. Then we shortly compare all three. The difference is only in the nature of the fractal dimension used: either using a box-counting, Hausdorff or Minkowski version of the dimension. Throughout we always (more or less silently) assume that we are dealing with sets such that the corresponding fractal dimensions exist. Whereas here all the dimensions are measure-based, we also remark that in \cite{BS} a 1D metric version of KPZ relation was proved in the context of dyadic multiplicative cascades. 

\paragraph{Expected box-counting version}\mbox{}\\
The first rigorous version of the KPZ relation was given in the work of Duplantier-Sheffield \cite{DS}, to which an interested reader can find a well-readable introduction in \cite{GAR}. Here the fractal dimensions for a fixed set $A$ on the Euclidean and on the quantum side are defined as follows (assuming they exist in the first place):
\begin{itemize}
\item Euclidean side: 
\[x(A) = \lim_{r \downarrow 0}\frac{\log \PP(B_r(z) \cap A \neq \emptyset)}{\log r}\] where we sample according the the uniform measure of the domain.
\item Quantum side:  \[\Delta(A) = \lim_{r \downarrow 0}\frac{\log \EE \mu_h(B_r^q(z) \cap A \neq \emptyset)}{\log r}\]
Here the quantum ball $B^q_r(z)$ of radius $r$ is defined as the largest Euclidean ball around $z$ for which the Liouville measure is not larger than $r$. 
\end{itemize}
With these notions the KPZ relation holds:
\begin{theorem}[Duplantier \& Sheffield]
Let $A$ be a deterministic (or field-independent) compact subset in the interior of some domain such that its Euclidean scaling exponent $x(A)$ exists. Let $\mu_\gamma$ be the Liouville measure on this domain with $0 \leq \gamma < 2$. Then we have that:
\begin{itemize}
\item the quantum scaling exponent $0 \leq \Delta(A) \leq 1$ exists and 
\item satisfies the so called KPZ formula:
\[x = (2-\gamma^2/2)\Delta + \gamma^2\Delta^2/2\]
\end{itemize}
\end{theorem}
\paragraph{Almost sure Hausdorff version}\mbox{}\\
Soon thereafter, Rhodes \& Vargas \cite{RV} published a version using slightly different notion for the fractal dimension. The proof of the respective KPZ relation can be made quite short \cite{GMC2}. As a basis for their definition of the quantum dimension, they use a measure-based Hausdorff dimension.
\begin{itemize}
\item On the Euclidean side we use the usual Hausdorff dimension. I.e. define the Hausdorff content 
\[H_\delta(A,r) = \inf \{\sum_i r_i^\delta: A \subset \cup_1^k B_i(r), r_i \leq r \}\]
Then the Hausdorff dimension is defined as
\[d_H(A) = \inf_\delta \{\lim_{r \downarrow 0} H_\delta(A,r) < \infty \}\]
\item For the quantum side, we define similarly the quantum Hausdorff content to be 
\[H_\delta^Q(A,r) = \inf \{\sum_i \mu(B_i(r_i))^\delta: A \subset \cup_1^k B_i(r), r_i \leq r\}\]
The quantum Hausdorff dimension is then given by 
\[q_H(A) = \inf_\delta \{\lim_{r \downarrow 0} H_\delta(A,r) < \infty \}\]
\end{itemize}

Then the following KPZ relation holds.
\begin{theorem}[Rhodes \& Vargas]
Let $A$ be a deterministic (or field-independent) compact subset in the interior of some domain. Let $\mu_\gamma$ be the Liouville measure on this domain with $0 \leq \gamma < 2$. Then, almost surely, the following KPZ formula holds:
\[d_H = (2+\gamma^2/2)q_H^q - \gamma^2q_H^2/2\]
where by $d_H$ and $q_H$ we denote respectively the usual and the quantum Hausdorff dimensions.
\end{theorem}

\paragraph{Expected Minkowski version}\mbox{}\\
To make the literature even more colourful, we introduce yet a third version of the dimension which also satisfies the KPZ relation. We use a version of the upper Minkowski dimension, which we will henceforth call just the Minkowski dimension.
 
There are many ways to define the Minkowski dimension, for us the most convenient version uses only fixed dyadic tiling \cite{BP}:

Consider a dyadic $2^{-n}$ Minkowski content of $A$ defined by:
\[M_\delta(A,2^{-n}) = \sum_{S_i \in \mathcal{S}_n}\II(S_i \cap A \neq \emptyset)l(S_i)^\delta\]
where $\mathcal{S}_n$ is the $n$-th level dyadic covering of the domain and $l(S_i)$ the side-length of the square $S_i$. Then we define the Minkowski dimension as
\[d_M(A) = \inf_\delta \{\limsup_{n \uparrow \infty} M_\delta(A,2^{-n}) < \infty \}\]

The corresponding quantum version is given by first defining the quantum dyadic $2^{-n}$ Minkowski content:
 \[M^Q_\delta(A,2^{-n}) = \sum_{S_i \in \mathcal{S}_n}\II(S_i \cap A \neq \emptyset)\mu(S_i)^\delta\]
 and then setting
 \begin{equation*}
q_M(A) = \inf_\delta \{\limsup_{n \uparrow 0} M^Q_\delta(A,2^{-n}) < \infty \}
\end{equation*}

It is clear that the definitions work nicely also for random sets, in which case the Minkowski contents will just be random variables. 

Moreover, it will also make sense to talk about the expected quantum Minkowski dimension, where in the definition of the Minkowski dimension, we just use the expectation of the dyadic Minkowski content w.r.t the measure. So, for deterministic sets we set for example:

\[q_{M,E}(A) = \inf_\delta \{\limsup_{n \uparrow \infty} \EE_h \left(M^Q_\delta(A,2^{-n})\right) < \infty \}\]

Notice that we take the expectation of each dyadic $2^{-n}$ Minkowski before the $\limsup$. Whereas this is less natural, it allows us to work only with first moment estimates and nevertheless provide upper bounds for the quantum Hausdorff dimension.

In the next section, we will prove the analogous KPZ relation for the expected quantum Minkowski dimension, the proof of which is shorter than for the other two notions:

\begin{proposition}
Let $A$ be a fixed (or field-independent) compact subset in the interior of some domain. Let $\mu_\gamma$ be the Liouville measure on this domain with $0 \leq \gamma < 2$. Then we have the following KPZ formula:
\[d_M = (2+\gamma^2/2)q_{M,E} - \gamma^2q_{M,E}^2/2\]
where by $d_M$ and $q_{M,E}$ we denote respectively the usual (upper) and the expected quantum Minkowski dimensions.
\end{proposition} 

\paragraph{Relations between the notions}\mbox{}\\
These three different notions of the quantum dimension and hence the KPZ relation all have different benefits: 

\begin{itemize}
\item Box counting version: it provides a notion of quantum balls having more physical content and is probably easiest to link to discretization of the field, and hence discrete models.
\item Almost sure Hausdorff version: whereas the box counting version is averaged over the field, here we have an almost sure relation; it also has the usual advantages and specificities with respect to the Minkowski dimension. However, it proved difficult to use for field-dependent sets.
\item Expected Minkowski: this is easiest to work with for both dependent and independent sets; one might say it is less natural, however it certainly has enough substance to give useful bounds on the Hausdorff dimension.
\end{itemize}

In the next section, we will also prove two relations between the expected quantum Minkowski and quantum Hausdorff dimensions. 

Firstly, we show that for deterministic and measure-independent sets we have the following relation: if the Euclidean Minkowski and Hausdorff dimensions of a set agree, then also its expected Minkowski dimension and Hausdorff dimension agree on the quantum side. This shows that we are not losing much in general by using the Minkowski version

Secondly, we show that on the quantum side the quantum Hausdorff dimension is almost surely smaller than the expected Minkowski dimension, even if the measured set depends on the field. This will allow us to prove results about the almost sure Hausdorff version, by fist proving them for the expected Minkowski dimension.

\paragraph{KPZ relation for dependent sets}\mbox{}\\
Notice that in all three theorems we require the sets in question to be either fixed or independent of the underlying measure. Hence it is natural to ask, to what extent the KPZ relation remains true for sets that depend on the measure. It comes out that there is no uniform theorem as for example Kaufman's theorem for dimension doubling in Brownian Motion. 

In fact, given that the KPZ relation stems from a multifractal behaviour \cite{GMC}, it is quite intuitive that for example fixed level sets should help us construct already a counterexample. The problem is that the precise counterexamples depend on the "sensitivity" of the definition and the intuitively clearest versions won't always work:

For almost sure Hausdorff dimension finding a counterexample is relatively easy. One just needs to look at $\gamma-$-thick points \cite{KAHANE} \cite{PERES} \cite{BAR2}, i.e. points such that $\lim_{r \downarrow 0} \frac{h_\epsilon(z)}{\log 1/r} = \gamma$. Their Hausdorff dimension is smaller than two, but they are of full measure on the quantum side, violating the usual KPZ relation.

For expected box-counting measure and the Minkowski dimension finding a counterexample is somewhat harder, as they are less sensitive. For example $\gamma-$ thick points, being dense, would have trivial dimensions on both sides. To produce a simple counterexample one needs to go one step further. We can still rely on the height of the field to produce a fractal as in \cite{PERES}, but we need to intersect this field-dependent fractal with a deterministic fractal to arrive at the "sensitivity" level of these definitions.

Now these previous examples might look unnatural - in some sense we were really trying to cook up counterexamples. Thus it would be interesting to find counterexamples where the measure-dependent sets are not a priori chosen to violate KPZ. This was exactly the aim of this paper: we look at the zero level lines and SLE$_\kappa$ flow lines given by the coupling of the GFF and the SLE and show that the expected Minkowski and almost sure Hausdorff versions of the KPZ relation do not hold for these sets. Thus, even for rather natural couplings the KPZ relation cannot be taken as given. 

\section{Expected Minkowski dimension: KPZ formula and relation to almost sure Hausdorff dimension}

\subsection{KPZ formula for expected Minkowski dimension}

In this subsection we will prove the following proposition:

\begin{proposition}[KPZ formula for expected Minkowski dimension]\label{KPZM}
Let $A$ be a fixed (or field-independent) compact subset in the interior of some domain. Let $\mu_\gamma$ be the Liouville measure on this domain with $0 \leq \gamma < 2$. Then we have the following KPZ formula:
\[d_M = (2+\gamma^2/2)q_{M,E} - \gamma^2q_{M,E}^2/2\]
where by $d_M$ and $q_{M,E}$ we denote respectively the usual (upper) and the expected quantum Minkowski dimensions.
\end{proposition}  

The proof is a simple consequence of the scaling properties that we state as a lemma. For the proof and slightly generalized versions, we refer to one of the many newer works on multiplicative chaos, including \cite{RRV} \cite{RV}, but also to \cite{DS} where it is approached slightly differently. 

\begin{lemma}[Scaling relation of Liouville balls]\label{scalinglemma}
Consider the Liouville measure $\mu = \mu_\gamma$ for $0 < \gamma < 2$. Then for any $q \in [0,1]$ and any fixed ball $B(r) \subset D$ of radius $\epsilon > r > 0$ at least at distance $\epsilon$ from the boundary, we have 
\[\EE\mu(B(r))^q \asymp r^{(2+\gamma^2/2)q - \gamma^2q^2/2}\]
where the implied constant depends on $q$.
\end{lemma}

\begin{remark}
If the distance of the ball is comparable to the boundary, one needs to be more careful as the exact scaling holds for the covariance kernel given by $\log_+ \frac{1}{\abs{x-y}}$ and the correction term of the Green's function starts playing a greater role near the boundary.
\end{remark}

\begin{proof}[Proof of proposition]\mbox{}
\paragraph{\small{Upper bound}}\mbox{}\\
Let $\delta > 0, 1 \geq q > 0$ be such that $d_M + \delta = (2+\gamma^2/2)q-\gamma^2q^2/2$. We want to show that $\limsup_{n} \EE M^Q_q(E,2^{-n}) < \infty$.  As the Minkowski dimension of $A$ is $d_M$, then for sufficiently large $n$ 
\[M_{d_M + \delta}(A,2^{-n}) \lesssim 2^{-n\delta/2}\]
Thus, for the same covering we get using the scaling relation of \ref{scalinglemma}, that 
\[\EE (M^Q_{q}(A,2^{-n}) \lesssim 2^{-n\delta/2}\]
Thus $q_{M,E} \geq q$. Now letting $\delta \downarrow 0$, we get the upper bound.

\paragraph{\small{Lower bound}}\mbox{}\\
The lower bound follows similarly. As $d_M$ is the Minkowski dimension for $A$, then for any $\delta > 0$, we have infinitely many $n \in \mathbb{N}$ such that $M_{d_M - \delta}(A,2^{-n}) > R$ for any $R > 0$. Now consider $1 \geq q > 0$ such that $d_M - \delta = (2+\gamma^2/2)q - \gamma^2q^2/2$. Then for all the same indexes $n$, we have $\EE M^Q_q(A,2^{-n}) > R$ and the lower bound follows.
\end{proof}

\begin{remark}
Notice that for the upper bound we could use an "almost sure" version of the Minkowski dimension. Indeed, from Markov's inequality 
\[\PP(M^Q_{q}(A,2^{-n})\geq 2^{-n\delta/4}) \leq 2^{-n\delta/4}\]
Now this sequence of probabilities is summable and thus by Borel-Cantelli the event only happens finitely often.Thus in fact almost surely $\limsup_n M^Q_{q}(A,2^{-n})=0$. 
\end{remark}

\begin{remark}
Also, it is easy to see that the same result holds for sets that are independent of the field. 
\end{remark}

\subsection{Relations between expected Minkowski and almost sure Hausdorff dimension}

In this section we bring out two results. First, for fixed (and field-independent) sets we conclude an agreement between the expected Minkowski and almost sure Hausdorff versions of the quantum dimension, given that there is agreement between the dimensions on the Euclidean side. Second, we prove an inequality for the quantum side holding even for dependent sets.

The first relation, as both the Hausdorff and Minkowski dimension satisfy the very same KPZ relation, is a straightforward corollary of the previous proposition: 

\begin{corollary}
Consider the Liouville measure for $0 \leq \gamma < 2$ in some domain. Suppose $A$ is deterministic (or field-independent) compact set in the interior of some domain, such that its Euclidean Minkowski and Hausdorff dimensions agree. Then also, its expected quantum Minkowski dimension and quantum Hausdorff dimensions agree. 
\end{corollary}

The second relation importantly also holds for sets that can depend on the measure:

\begin{proposition}\label{MinHff}
Consider the Liouville measure with $0 \leq \gamma < 2$ in some domain. For any random set coupled with the field, the quantum Hausdorff dimension is almost surely bounded above by the expected quantum Minkowski dimension.
\end{proposition}

To prove this, first notice that in fact we could equally well use squares instead of balls in our definition of the (quantum) Hausdorff dimension. 

\begin{proof}
Suppose that with positive probability $p > 0$ the quantum Hausdorff dimension of the set $A$ satisfies $q_H(A) > \delta$. Then also 
\[\PP\left(\lim_{n \uparrow \infty} H^Q_\delta(A,2^{-n}) = \infty\right) = p\]
where we use squares instead of balls in the covering. But now every covering used in the Minkowski dimension also provides a suitable covering whose content must be larger than $H^Q_\delta(A,2^{-n})$. Hence it follows that  
\[\PP\left(\liminf_{n \uparrow \infty} M^Q_\delta(A,2^{-n}) = \infty\right)  \geq p\]
Now fix some $R > 0$ large and define the event 
\[E_{N,R} = \{M^Q_\delta(A,2^{-n}) > R \text{ for all } n \geq N\}\]
The events $E_{N,R}$ are increasing in $N$ and
\[\bigcup E_{N,R} = \{\liminf_{n \uparrow \infty} M^Q_\delta(A,2^{-n}) = \infty\}\]
Thus by countable additivity there is some $N_R$ such that $\PP(E_{N_R,R}) > p/2$. But then for all $n > N_R$
\[\EE(M^Q_\delta(A,2^{-n})) \geq Rp/2\]
And thus 
\[\limsup_{n \uparrow \infty} \EE \left(M^Q_\delta(A,2^{-n})\right) \geq Rp/2\]
But $p > 0$ was fixed and we can pick $R$ arbitrarily large. Therefore 
\[\limsup_{n \uparrow \infty} \EE \left(M^Q_\delta(A,2^{-n})\right) = \infty\]
and $q_{M,E}(A) \geq \delta$. As this holds for all $\delta$ with $\PP(q_H(A) > \delta) > 0$, we have the claim.
\end{proof}

\begin{remark}
Notice that we do indeed a proof. Namely, we have no scaling result similar to lemma \ref{scalinglemma} at our disposal. So we do not a priori know that the Hausdorff and Minkowski contents scale well on the quantum side. Secondly, more direct approaches are limited by the fact that our definition of the Minkowski dimension involved an expectation inside the $\limsup$. 
\end{remark}

\section{Almost sure Hausdorff dimension of the zero level line does not satisfy the KPZ relation}
In this section we show that the expected Minkowski and almost sure Hausdorff versions of the usual KPZ relation do not hold for zero level lines. From now on we fix our underlying domain to be the upper half plane.

\begin{proposition}\label{LevL}
Consider the Liouville measure $\mu_\gamma$ with $0 \leq \gamma < 2$ in the upper half plane. The expected quantum Minkowski dimension of the zero level line drawn up to some finite stopping time satisfies $q_{M,E} \leq \frac{3}{4+\gamma^2}$. Hence the usual KPZ relation does not hold. 
\end{proposition}

By using proposition \ref{MinHff}, we have a straightforward corollary:

\begin{corollary}
Almost surely the quantum Hausdorff dimension of the zero level line drawn up to some finite stopping time is bounded from above by $\frac{3}{4+\gamma^2}$ and hence the usual KPZ relation is not satisfied for quantum Hausdorff dimension.
\end{corollary}

\begin{remark}
In fact, this proposition can also be seen as a straightforward corollary of the later work on flow lines by setting $\kappa = 4$. In fact, we then also confirm that the expected Minkowski dimension of the zero level line is equal to $q = \frac{3}{4+\gamma^2}$. However, the proof here is much shorter and simpler in spirit. The underlying intuition is that near the zero level line the field is lower and this intuition can be nicely expressed with rigour.
\end{remark}

We start with a key lemma that replaced the usual scaling lemma and gives the scaling of quantum balls around points on the zero level line:

\begin{lemma}\label{LL1}
Fix a zero level line drawn up to some finite stopping time. Let $S$ be a dyadic square of side-length $l(S)$ intersecting the zero level line. Denote by $h$ the Gaussian free field in this split domain and by $\mu$ the corresponding Liouville measure with $\gamma < 2$. Then we have that $\EE_h\left(\mu_h(S)\right) \lesssim l(S)^{2+\gamma^2/2}$
\end{lemma}

\begin{proof}
As usual in working with the Liouville measure, it is cleaner to work with a regularized field. From theorem \ref{Liouville} we know $\delta_n = 2^{-n}$ regularized fields converge to the Liouville measure. Hence, we can write
\[\mu_h(S) = \lim_{\delta_n \downarrow 0} \mu_{h_{\delta_n}}(S) = \lim_{\delta_n \downarrow 0} \int_{S} \delta^{\gamma^2/2} e^{\gamma h_{\delta_n}(z)}dz \]
Recall from definitions preceeding \ref{Liouville} that the regularized field $h_{\delta_n}(z)$ is a Gaussian field, defined by taking circle averages of the GFF. It is defined nicely point-wise. Its mean is given by the bounded harmonic SLE-measurable correction term described in section 2.1, and the covariance kernel is described by the regularized Green's function of the slit domain.
\[G_{\delta_n}(x,y) = \log \frac{1}{\delta_n \vee \abs{x-y}} + \tilde{G}_{\delta_n}(x,y)\]
Here $\tilde{G}_{\delta_n}(x,y)$ is the harmonic extension of the function equal to $-\log \frac{1}{\delta_n \vee \abs{x-y}}$ when one of the points is on the boundary of the domain. Notice that if at least one of $x, y$ is of distance $\delta_n$ from the boundary, then $\tilde{G}_{\delta_n}(x,y) = \tilde{G}(x,y)$ where the latter is the harmonic correction term for the usual Green's function. This is useful, as we know that $\tilde{G}(x,x) = \CR(x,H_t)$ where the latter denotes the conformal radius of the point $x$ for the slit domain. 

Now we can write the GFF $h$ as a sum of a zero-boundary GFF $h^0$ and the bounded harmonic correction term $C_h$ that can be defined to be zero on the SLE (see discussion after the statement on theorem \ref{thmCL}. So proceed to write
\[\EE_h\left(\mu_h(S)\right) = \lim_{\delta_n \downarrow 0}\EE_{h}\left(\int_{S} \delta_n^{\gamma^2/2} e^{\gamma h^0_{\delta_n}(z)+\gamma C_h}dz\right)\]
Firstly notice that as the harmonic correction is uniformly bounded by a constant, it will only influence the expectation by a bounded constant and thus we can henceforth neglect the term $\gamma C_h$ by absorbing it in some multiplicative constant. Thus we want to bound
\[\EE_{h}\left(\int_{S} \delta_n^{\gamma^2/2} e^{\gamma h^0_{\delta_n}(z)}dz\right)\]
We will split the integral into two:
\begin{enumerate}
\item the part that is at least of distance $\delta_n$ off the curve
\item the curve together with its $\delta_n$ neighbourhood
\end{enumerate}
For the first part, start by taking the expectation inside the integral (everything is nicely bounded). Then using exponential moments for Gaussian random variables, we have the following estimate for the integrand:
\begin{equation}
\label{expBound}
\EE_h \delta_n^{\gamma^2/2} e^{\gamma h^0_{\delta_n}} \leq \CR(z,H_t)^{\gamma^2/2}
\end{equation}

Recall that the conformal radius satisfies $\CR(z,H_t) \asymp d(z,H_t)$ where $d(z,H_t)$ is the distance from the boundary. But $d(z,H_t) \leq l(S)$ and hence we get a bound of $O(l(S))^{\gamma^2/2}$. Thus integrating over the whole square (minus the $\delta_n$ neighbourhood) we get a contribution of $O(l(S)^{2+\gamma^2/2})$.

Now we treat the part near the curve. We could use Kahane convexity inequalities \cite{KAHANE} or a global argument as in \ref{FLZ}. However, it follows also elementarily by using bare hands. Start again by taking the expectation inside the integral. Then we need to bound the variance of $h^0_{\delta_n}(z)$. By the definition of the GFF in $H_t$ it is given by integrating 
\[\int_{H_t \times H_t}G_{\delta_n}(x,y)\rho^z_{\delta_n}(x)\rho^z_{\delta_n}(y)dxdy\]
where by $\rho^z_{\delta_n}$ we denote the distribution giving unit mass to the circle of radius $\delta_n$ around the point $z$. 

But $G(x,y) \geq G_{\delta_n}(x,y)$ and hence the variance is bounded by 
\[\int_{\HH \times \HH}G(x,y)\rho^z_{\delta_n}(x)\rho^z_{\delta_n}(y)dxdy\]
i.e. by that of the $\delta_n$ regularized GFF in $\HH$. But this we can calculate as above to get 
\[\EE_h \delta_n^{\gamma^2/2} e^{\gamma h^0_{\delta_n}} \leq \CR(z,\HH)^{\gamma^2/2}\]
In particular, integrating over the $\delta_n$ neighbourhood of the curve, as the Hausdorff dimension of SLE$_4$ is $3/2$ \cite{BEF}, we may bound this part with $o(\delta_n^{1/3})$

Thus 
\[\EE_h\left(\int_{S} \delta_n^{\gamma^2/2} e^{\gamma h_{\delta_n}(z)}dz\right) \lesssim (l(S))^{2+\gamma^2/2} + o(\delta_n^{1/3})\]
and letting finally $\delta_n \downarrow 0$, we get
\[\EE_h\left(\mu_h(S)\right) \lesssim l(S)^{2+\gamma^2/2}\]

\end{proof}
Now we are ready to attack the proposition:
\begin{proof}[Proof of proposition]

We will sample the GFF as above: we first sample an $SLE_4$ up to some finite stopping time, then the field in the slit domain with its bounded harmonic correction term. 

Now, we know that the Minkowski dimension of the $SLE_4$ curve is $3/2$ \cite{ROHDESCHRAMM,BEF}. Thus for any $\delta > 0$ we can cover it with $O(r^{-3/2-\delta})$ dyadic squares $S_i, i \in \mathcal{I}$ of radius $r=2^{-n}$. Fix $q < 1$ to be defined later. 

By linearity of expectation we can write
\[\EE_h M_q^Q(A,r) = \EE_h \left(\sum_{i \in \mathcal{I}} \mu_h(S_i)^q\right) = \sum_i \EE_h\left(\mu_h(S_i)^q\right)\]
Now by lemma \ref{LL1}, $\mu_h(S_i)$ is an integrable random variable with respect to the randomness of the GFF $h$. Hence as $q \leq 1$, we can use Jensen's inequality for the concave function $x^q$ to get 
\[\EE_h \left(\mu_h(S_i)^q\right) \leq \left(\EE_h \mu_h(S_i)\right)^q\]
But using lemma \ref{LL1} again, we have for any ball $S_i$
\[\left(\EE_h \mu_h(S_i)\right)^q \lesssim r^{q(2+\gamma^2/2)}\]
and so 
\[\EE_h \left(\sum_{i \in \mathcal{I}}\mu_h(S_i)^q\right) \lesssim r^{-3/2 - \delta + q(2+\gamma^2/2)} \]
Choosing $q = (1+\delta)\frac{3}{4+\gamma^2}$ we thus have 
\[\EE_h M_q^Q(A,r) \lesssim r^{\delta/2}\]
It follows that $q_M \leq q$ and by letting $\delta \downarrow 0$, we see that $q_M \leq \frac{3}{4+\gamma^2}$. 
\end{proof}

\section{Winding of SLE$_\kappa$}
In this section we find the exponential moments for the winding of chordal SLE curves conditioned to pass nearby a fixed point. 

\subsection{Introduction and results}
The winding we study in this section is in exact correspondence with the additional correction term in the flow line coupling of theorem \ref{thmC}.

\begin{definition}\label{defW}
Consider a chordal SLE$_\kappa$ in the upper half plane. The winding $w(z)$ around a point $z$ up to time $T$ is given by the argument of $f_T'(z)$. For $4 < \kappa < 8$, we need to take $\lim_{t\uparrow T'} f_t(z)$ with $T' = T \wedge \tau(z)$, where $\tau(z)$ is the first time $z$ is separated from the infinity by the SLE curve. If we do not mention the time $T$, we consider the entire SLE curve. 
\end{definition}

Notice that as $\arg g'(z)$ is the imaginary part of an analytic function $\log g'(z)$, it is a harmonic function off the curve itself. We fix the logarithm by requiring it to be continuous in the slit domain and tend to $0$ at infinity \cite{ZIPPER}. The basic intuition behind winding is that whereas $\abs{g'(z)}$ measures the distortion of the length under $g$, then $\arg (g'(z))$ measures the angular distortion. Very near the curve, this distortion is given by unwinding the SLE curve back to zero. One can also think that this definition of winding gives the amount that a curve from the infinity needs to wind to access the point $z$. Asymptotically near the curve, this version of winding should coincide with the geometric winding up to some bounded constants \cite{WIND}.

The coupling of GFF and SLE gives the average winding of SLE over the randomness of the SLE. Here, we prove the following more precise result, calculating the winding around any point depending on its distance to the SLE curve. Recall that we are working with the chordal SLE in the upper half plane. 

\begin{theorem}\label{thmW}

Let $\CR_0$ be the conformal radius of a fixed point $z_0$ in the upper half plane. Fix $0 <  \kappa < 8$ and let $\tau$ be the time that SLE first cuts $z_0$ from infinity. Denote by $H_\tau$ the SLE slit domain component containing $z_0$. Then, for $\epsilon > 0$ sufficiently small, conditioned on $\CR(z_0,H_\tau) \in [\epsilon, C\epsilon]$ with $C > 1$, the exponential moments of the winding $w(z_0)$ around the point $z_0$ are given by $e^{\lambda w(z_0)} \asymp \epsilon^{-\lambda^2\kappa/8}$, where the implied constants depend on $\kappa, \lambda$ and for fixed $\kappa$ can be chosen uniform for $\abs{\lambda} < \lambda_0$ for any choice of $\lambda_0 > 0$.

\end{theorem}

\begin{remark}
We have defined the winding in the upper half plane and also stated the theorem in there. However, as defining the chordal SLE in a different nice (for example rectifiable, smooth Jordan boundary) domain would involve conjugations by analytic maps that extend to the boundary and have non-zero derivative on the boundary almost everywhere, the winding in any other such domain will be the same up to a uniformly bounded additive error. Hence, as we determine exponential moments up to multiplicative constants, the theorem \ref{thmW} holds also for the chordal SLE in all nice domains and in particular in the unit disc.
\end{remark}

\begin{remark}
By following the proof carefully, we actually get slightly more: we get that the winding is given by a Gaussian of variance $-\frac{\kappa}{4}\log {\epsilon}$ plus different error terms. The dependence relations between these error terms are a bit delicate and that is also the reason why we chose the wording above, which, needless to say, is sufficient for our applications. For $\kappa = 8$ the curve is space filling and winding should actually give the Gaussian free field, seen as a "height function".
\end{remark}

\paragraph{Comparision to Schramm's study on winding}\mbox{}\\

In this paragraph, we will shortly discuss how this result relates to Schramm's work on winding in his seminal paper \cite{SCHRAMM}. First, Schramm actually studied the winding of radial SLE around its endpoint zero and the variance was approximated by a Gaussian of variance $-\kappa\log \epsilon$, when the tip was $\epsilon$-close to zero. However, in our case we have a Gaussian of variance $-\kappa/4 \log \epsilon$. This seems to be in agreement with predictions by Duplantier (see e.g. \cite{DUPN}, ch. 8), where radial SLE ought to correspond to a one-arm event and chordal SLE conditioned to be close to a point - we think - could correspond to a two-arm event. Intuitively for $\kappa$ small, one could argue that in the chordal case you just pass from one or other side of the point, whereas in the radial case you might still do a turn before finally hitting zero, thus causing a difference in variances.

Also, one needs to remark that notions of winding in \cite{SCHRAMM} and here differ. Schramm is really looking at the geometric winding number around zero, this is given by the argument of the tip of the curve, when the argument is chosen to be continuous along this curve. We, however, use the definition of \cite{ZIPPER} that gives the GFF-SLE couplings above. As explained above and as used in physics literature \cite{WIND}, these two notions should asymptotically agree up to bounded additive errors. What we can confirm is that indeed a few line of calculations show that in the radial case around point zero, the concept used here would give a Gaussian of variance $-\kappa \log \epsilon$, in agreement with Schramm's result. 

Finally, there is the question whether Schramm's nice geometric approach could have helped the technical work to follow. It does not seem to be the case, as his method in some sense only helps to relate the winding of the curve to the behaviour of the driving process. Due to conditioning, in our case the work is actually in studying the behaviour of the driving process resulting from conditioning.

\subsection{Proof of the theorem}

To start attacking the theorem, we need a lemma to translate the question that of diffusion processes and rewrite the geometric conditioning of SLE curves in terms exit times of a certain diffusion process:

\begin{lemma}\label{lemB}
Consider the chordal SLE$_\kappa$ in the upper half plane with $0 < \kappa < 8$ and set $\CR_0 = CR(z_0,\mathbb{H})$. Then, the conformal radius $CR(z_0,H_\tau)$ is equal to $\CR_0 e^{-\tau}$, where $\tau$ is the first time when the SLE curve cuts $z_0$ from the infinity, and also the first exit time of a diffusion $\alpha_s$ in $(0, 2\pi)$ satisfying the following equation:
\begin{equation}\label{eqD}
d\alpha_s = \sqrt{\kappa}dB_s + \frac{\kappa - 4}{2}\cot \frac{\alpha_s}{2}ds
\end{equation}
Moreover, the winding around $z_0$ is given by $w(z_0) = \int_0^\tau \cot \frac{\alpha_s}{2}$.
\end{lemma}

\begin{remark}
This lemma stems from the first moment argument in \cite{BEF}. The basic strategy is the following: we transform our chordal SLE in $\mathbb{H}$ to a process in $\mathbb{D}$ for which the image of $z_0$ is fixed to the origin, then pick a convenient time change, and study the process induced for the driving Brownian motion. We only need slightest adjustments, but for the convenience of the reader, the proof is still provided in the appendix. Notice that in case of $\kappa > 4$ the exit time of the diffusion corresponds to the first time when the point $z_0$ is separated by the curve from the infinity.
\end{remark}

\begin{proof}[Proof of the Theorem \ref{thmW}]

From lemma \ref{lemB}, we see that conditioning on 
\[CR(z_0,H_\tau) \in [\epsilon, C\epsilon]\]
is equivalent on conditioning the corresponding diffusion to exit $(0,2\pi)$ during the time interval 
\[(\log \CR_0 + \log \frac{1}{\epsilon} - \log C, \log \CR_0 + \log \frac{1}{\epsilon}]\] 
Recall that $\tau$ is also the first exit time for the diffusion and set $T = \log \frac{1}{\epsilon} + \log \CR_0 - \log C$. Then it remains to show that conditioned on $\tau \in [T,T+c]$, we have 
\[\EE\left(\exp(\lambda \int_0^\tau \cot \frac{\alpha_s}{2})\right) \asymp e^{T\lambda\kappa/8}\]
with uniform constants for $\abs{\lambda} < \lambda_0$ for any choice of of $\lambda_0 > 0$.

We will do this in several steps: first, the main term of the theorem comes from the conditioned diffusion up to time $T-10$. By gaining control on eigenfunction expansions of survival probability, we show that this part is more or less stationary and absolutely continuous with respect to the process conditioned to everlasting survival. Thereafter, we have to control the rest. As the behaviour of the diffusion starts to change and we need to opt a different strategy. The more dangerous part is the very end and we want to handle it (for $\kappa \neq 4$) independently of the main term, thus we introduce yet another subdivision at time $T-9$. These error terms are then controlled using probabilistic arguments. 

\begin{center}
    \includegraphics[width = 0.8\textwidth, keepaspectratio]{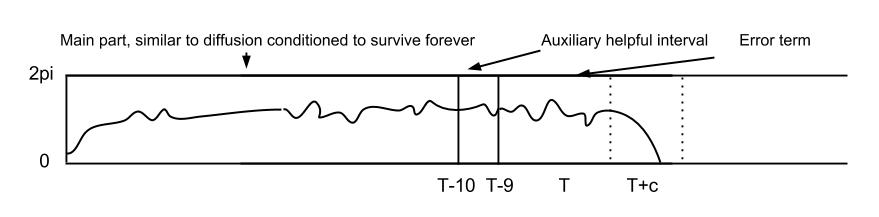}
    \end{center}

\subsubsection{Boundary growth of eigenfunctions for the Green operator} 
For the main part, the key is obtaining good estimates of the survival probability of the diffusion. To do this, we gain control over the eigenfunction expansion of a related differential operator.

Although inside the interval everything about our diffusion \eqref{eqD} is nice and smooth, we have to be cautious because the drift term becomes singular at both ends of the interval. Recall that when one considers one-dimensional diffusions on its natural scale - basically turning it into a martingale - then the speed measure represents the time-change with respect to a standard Brownian motion. In our case this speed measure is seen to be 

\[m(dx) = \sin^{2-\frac{8}{\kappa}}\frac{x}{2}dx\]

which is integrable over the interval $[0, 2\pi]$ only for $\kappa > 8/3$. Still everything will work out nicely. 

The first step is to find the Green's function corresponding to the diffusion process. This can be done either purely analytically \cite{Yosida}, or probabilistically by first finding the Green's function on natural scale of the diffusion and then transforming back to the initial setting \cite{BASS}.

As a result we see that the Green's function is given by
\begin{equation}\label{GF}
G(x,y) = \left\{ \begin{array}{rcl}
\frac{s(x)(s(2\pi) - s(y))}{s(2\pi)} & \mbox{for}
& x \leq y \\ \frac{s(y)(s(2\pi) - s(x))}{s(2\pi)} & \mbox{for} & x>y
\end{array}\right.
\end{equation}
where $s(x)$ is a scale function of the diffusion given by 
\[s(x) = \int_0^x \sin^{\frac{8}{\kappa}-2}\frac{u}{2}du\]

Now consider the corresponding integral operator:
\[Gf(x) = \int G(x,y)f(y)m(dy)\]
on $L^2(I,m(dx))$. A direct calculation shows that this satisfies the conditions of a Hilbert-Schmidt integral operator, i.e. its $L^2[(I,m(dx))\times(I,m(dx))]$ norm of the kernel is finite. Thus using Hilbert-Schmidt expansion theorem and Krein-Rutman theory \cite{BASS}, we have a complete orthonormal system of eigenfunctions $\phi_i(x)$ and corresponding growing eigenvalues $\lambda_i$ such that $0 < \lambda_0 < \lambda_1 \leq \lambda_3 \leq ... < \infty$. We remark that this would also follow by just considering the corresponding Sturm-Liouville problem: even though the problem is not regular at endpoints, the expansion still applies.

Notice also that as the corresponding diffusion (or its generator) has $C^2$ regularity inside any compact interval of $(0,2\pi)$, the eigenfunctions are also at least $C^2$ in these respective intervals. Moreover, by writing out the eigenfunction expansion for the Green's function itself and using Bessel inequality, we see that eigenvalues don't grow too hastily:

\begin{equation}
\label{EVG}
\sum_{i=0}\lambda_i^{-2} < \infty
\end{equation}
Next we would like to get a good control on individual eigenfunctions also near the boundary. An explicit calculation shows that $\lambda_0 = 1-\frac{\kappa}{8}$ and up to a normalization constant $\phi_0(x)$ is equal to $\sin^{\frac{8}{\kappa}-1}\frac{x}{2}$. In fact we will set $\phi_0(x) = \sin^{\frac{8}{\kappa}-1}\frac{x}{2}$ to ease some subsequent calculations.

For other eigenfunctions, we need some more work. As a first step we can use Cauchy-Schwartz on $\phi_i(x) = \lambda_i G\phi_i(x)$, to obtain
\begin{equation}\label{GR}
\begin{split}
\abs{\frac{1}{\lambda_i}\phi_i(x)} = \abs{G\phi_i(x)} &= \abs{\int G(x,y)\phi_i(y)m(dy)} \\
&\leq (\int G^2(x,y)m(dy))^{1/2}(\int \phi_i(y)^2m(dy))^{1/2} \\ 
&\lesssim 1
\end{split}
\end{equation}
or in other words $\phi_i(x) \lesssim \lambda_i$, where the implied constant does not depend on $i$.

However, this is not yet enough for our purposes. We need to show that the boundary growth of other eigenfunctions is at least of the same order than that of the first eigenfunction $\phi_0$. Thus we define for all $i \in \NN$
\[g_i(z) = \frac{\phi_i(x)}{\phi_0(x)}\]
and study its behaviour. We prove two lemmas about $g(z)$. First we show that all eigenfunctions scale similarly near the boundary or in other words:

\begin{lemma}\label{Cl1}
For all $i \in \NN$, we have 
\begin{equation*}
g_i(x) \lesssim \lambda_i^m
\end{equation*}
for some universal $m$.
\end{lemma}

Then we go on to push this control a step further to show that the boundary growth of other eigenfunctions is in fact even nicer:

\begin{lemma}\label{Cl2}
For all $i \in \NN$, we have 
\begin{equation*}
g_i'(x) \lesssim \lambda_i^{m+1}\sin \frac{x}{2}
\end{equation*}
where the implied constant does not depend on $i$.
\end{lemma}

\begin{proof}[Proof of lemma \ref{Cl1}]
To prove the first lemma, notice that it is enough to show the claim near $x = 0$, as firstly by \eqref{GR} and the fact that $\phi_0$ does not vanish inside the interval we know that the claim holds trivially in any compact subinterval of $(0,2\pi)$ and secondly, our diffusion is symmetric with respect to $\pi$ and thus boundary behaviour is the same near $0$ and $2\pi$.

Now the key is to notice that the Green's function is actually much more regular than needed for being in $L^2(I,m(dx))$. For example from $G_y(x) \lesssim \phi_0(x)$ it already follows that the Green's function lies in $L^1(I,m(dx))$. 

Our next aim is to use a bootstrap the scaling of the eigenfunctions $\phi_i(x)$, by improving step by step on the Cauchy-Schwarz in \eqref{GR}.
In this respect, consider the following expression for $x$ near $0$ and for $a \geq 0$:
\[z(x,a) = \int_0^{2\pi} G(x,y)^2\sin^{2a}\frac{y}{2}m(dy)\]

\begin{claim}\label{Cl3}
$z(x,a) \lesssim \max(\sin^{\frac{8}{\kappa}+2a+1}\frac{x}{2},\sin^{2(\frac{8}{\kappa}-1)}\frac{x}{2})$
\end{claim}

Using this claim, it is easy to improve step by step on the regularity of the eigenfunctions and to prove the lemma. 

Indeed, notice that in \eqref{GR} the first term on the RHS is given by 
\[z(x,0)^{1/2} \lesssim \lambda_i\sin^{\frac{4}{\kappa}+1/2}\frac{x}{2}\]
and thus it follows that $\phi_i(x) \lesssim \lambda_i\sin^{\frac{4}{\kappa}+1/2}\frac{x}{2}$. Notice that for $\kappa \geq 8/3$ we could hence stop here, as $\frac{4}{\kappa}+1/2 \geq \frac{8}{\kappa}-1$ and we already have the statement of the lemma. For smaller $\kappa$ consider the following bootstrap:

Suppose that we already know that $\phi_i(x) \lesssim \lambda_i^k\sin^{\frac{4}{\kappa}-1+a}\frac{x}{2}$. Then using a similar strategy as in \eqref{GR}, we could write using claim \ref{Cl3}
\begin{equation*}
\begin{split}
\abs{\frac{1}{\lambda_i}\phi_i(x)} = \abs{G\phi_i(x)} &= \abs{\int G(x,y)\phi_i(y)m(dy)} \\
&\leq \lambda_i^k(\int G^2(x,y)\sin^{2a}\frac{y}{2}m(dy))^{1/2}(\int \sin^{\frac{8}{\kappa}-2}\frac{y}{2}m(dy))^{1/2} \\ 
& = O(\lambda_i^k z(x,a)^{1/2}))\\
&\lesssim \lambda_i^k\max(\sin^{\frac{4}{\kappa}+a+\frac{1}{2}}\frac{x}{2},\sin^{\frac{8}{\kappa}-1}\frac{x}{2})
\end{split}
\end{equation*}
and thus $\phi_i(x) \lesssim \lambda_i^{k+1} \max(\sin^{\frac{4}{\kappa}+a+\frac{1}{2}}\frac{x}{2},\sin^{\frac{8}{\kappa}-1}\frac{x}{2})$.
Thereby we can improve on the boundary scaling $m-1$ times until we get $\phi_i(x) \lesssim \lambda_i^m\phi_0(x)$ as needed, whereas the implied constants have been independent of $i$.

Hence to prove the lemma we just need to prove the claim above.

\begin{proof}[Proof of claim \ref{Cl3}]
Using the form of the Green's function, we can bound 
\[z(x,a) = \int_0^{2\pi} G(x,y)^2\sin^{2a}\frac{y}{2}m(dy)\]
by the following:
\begin{equation*}
\begin{split}
z(x,a) \lesssim \max(&\int_0^x s(y)^2(s(2\pi)-s(x))^2\sin^{2a}\frac{y}{2}m(dy), \\
& \int_x^{2\pi} s(x)^2(s(2\pi)-s(y))^2\sin^{2a}\frac{y}{2}m(dy))
\end{split}
\end{equation*}
This can be further simplified to
\[z(x,a) \lesssim \max(\int_0^x s(y)^2\sin^{2a}\frac{y}{2}m(dy), s(x)^2)\]
Inserting now the definitions of the scale function and the speed measure, this gives us for $x$ small:
\[z(x,a) \lesssim \max(\sin^{\frac{8}{\kappa}+2a+1}\frac{x}{2},\sin^{2(\frac{8}{\kappa}-1)}\frac{x}{2})\]
Thus the claim \ref{Cl3} and the proof of lemma \ref{Cl1} follow.
\end{proof}
\end{proof}

The second lemma improves on this multiplicative regularity. And to prove it, we need to go back to the generator of the diffusion and use the fact that any eigenfunction of the Green's operator is also an eigenfunction of the generator \cite{Yosida}.

\begin{proof}[Proof of lemma \ref{Cl2}]
From the previous claim, we know than we can write $\phi_i(x) = \phi_0(x)g_i(x)$ for $g_i = O(\lambda_i^m)$. Notice also that $g_i$ has $C^2$ regularity inside any compact interval of $(0, 2\pi)$ as both $\phi_i$, $\phi_0$ have this regularity and $\phi_0=\sin^{\frac{8}{\kappa}-1} \frac{x}{2}$ is non-zero inside the whole interval.

Now every $\phi_i$ is also an eigenfunction of the generator of the diffusion. This can be stated in the Sturm-Liouville form:
\[\left(\frac{\kappa}{2}\sin^{2-\frac{8}{\kappa}}\frac{x}{2}\phi_i'(x)\right)'=\lambda_i\sin^{2-\frac{8}{\kappa}}\frac{x}{2}\phi_i(x)\]
Replacing now $\phi_i(x) = \phi_0(x)g_i(x)$, using the fact that $\phi_0(x)$ is an eigenfunction, we can calculate inside any compact interval of $(0,2\pi)$:
\[\frac{\kappa}{2}\sin^{2-\frac{8}{\kappa}}\frac{x}{2}\phi_0'(x)g_i'(x)+\left(\frac{\kappa}{2}\sin^{2-\frac{8}{\kappa}}\frac{x}{2}\phi_0(x)g_i'(x)\right)'=(\lambda_i-\lambda_0)\sin^{2-\frac{8}{\kappa}}\frac{x}{2}\phi_0(x)g_i(x)\]
Plugging in the exact form of $\phi_0(x)$ and a few calculations, we have:
\[2\cos\frac{x}{2}g_i'(x)+\frac{\kappa}{2}\sin\frac{x}{2}g_i''(x) =(\lambda_i-\lambda_0)\sin \frac{x}{2}g_i(x)\]

Thus we obtain the following Sturm-Liouville form for $g_i(x)$, which holds at least inside any compact of $(0,2\pi)$.
\[ \left(\frac{\kappa}{2}\sin^{\frac{8}{\kappa}}\frac{x}{2}g_i'(x)\right)' = (\lambda_i-\lambda_0)\sin^{\frac{8}{\kappa}}\frac{x}{2}g_i(x)\]

But now $g$ is bounded and $C^2$, the right hand side can be nicely integrated up to any $\epsilon > 0$ and we get 
\begin{equation}
\label{SL}
\left[\frac{\kappa}{2}\sin^{\frac{8}{\kappa}}\frac{x}{2}g'(x)\right]_\epsilon^{x_0} = (\lambda_i-\lambda_0)\int_\epsilon^{x_0}\sin^{\frac{8}{\kappa}}\frac{x}{2}g(x)dx
\end{equation}
We first claim the following:  
\begin{claim}
\label{Cl4}
As $\epsilon \downarrow 0$ we have 
\[\left[\frac{\kappa}{2}\sin^{\frac{8}{\kappa}}\frac{x}{2}g'(x)\right](\epsilon) = o(1)\] 
\end{claim}

We know that $\phi_i(x), g_i(x)$ are $C^2$ inside any compact interval of $(0,2\pi)$. Thus we can differentiate $\phi_i(x) = \phi_0(x)g(x)$ and using the triangle inequality write
\begin{equation}
\label{TRIANGLE}
\abs{\phi_0(x)g'(x)} \leq \abs{\phi_i'(x)} + \abs{\phi_0'(x)g(x)}
\end{equation}
For the second term of the RHS, we know that $\phi_0(x)=\sin^{\frac{8}{\kappa}-1}\frac{x}{2}$ and from lemma \ref{Cl1} we know that $g_i = O(\lambda_i^m)$. Hence the second term is of order $O(\lambda_i^m\sin^{\frac{8}{\kappa}-2}\frac{x}{2})$. To get a bound on the first term of the RHS consider again the integral equation satisfied by eigenfunctions:
\[\frac{1}{\lambda_i}\phi_i(x) = \int G(x,y)\phi_i(y)m(dy)\]
Now $\phi_i(x)$ is differentiable inside compacts of $(0, 2\pi)$, and also the Green's function $G(x,y)$ is
differentiable unless $x=y$, at which point it is both left and right-differentiable but these derivatives have a finite gap between them. Thus we can differentiate both sides to get:
\[\frac{1}{\lambda_i}\phi_i'(x) = \int \frac{\partial}{\partial x}G(x,y)\phi_i(y)m(dy)\]
Plugging in the form of the Green's function shows that the RHS can be bounded by $O(\lambda_i^m\sin^{\frac{8}{\kappa}-2}\frac{x}{2})$ and thus $\abs{\phi_i'(x)} =  O(\lambda_i^{m+1}\sin^{\frac{8}{\kappa}-2}\frac{x}{2})$.

Thus we see that in the triangle inequality \eqref{TRIANGLE}, the whole of RHS is of order \[O(\lambda_i^{m+1}\sin^{\frac{8}{\kappa}-2}\frac{x}{2})\]
In particular this must hold for the LHS, i.e. we have 
\[\abs{\phi_0(x)g'(x)} = O(\lambda_i^{m+1}\sin^{\frac{8}{\kappa}-2}\frac{x}{2})\]

To prove the claim, recall that $\phi_0(x) = \sin^{\frac{8}{\kappa}-1}\frac{x}{2}$. Hence, as $8\kappa > 1$, it follows that
\[\left[\frac{\kappa}{2}\sin^{\frac{8}{\kappa}}\frac{x}{2}g'(x)\right](\epsilon) = O(\lambda_i^{m+1}\epsilon^{\frac{8}{\kappa}-1})=o(1)\]
and thus our claim \ref{Cl4} follows.

Finally return to \eqref{SL}. The absolute value of the right hand side can be bounded by $O(\lambda_i^{m+1}\sin^{\frac{8}{\kappa}+1}\frac{x}{2})$ using lemma \ref{Cl1}. From our recent claim we know that by letting $\epsilon \downarrow 0$, only the term $\frac{\kappa}{2}\sin^{\frac{8}{\kappa}}\frac{x}{2}g'(x_0)$ survives. Thus get the claimed derivative bound: 
\[\left(\frac{\phi_i(x_0)}{\phi_0(x_0)}\right)'=g'(x_0) \lesssim \lambda_i^{m+1}\sin\frac{x_0}{2}\]
\end{proof}

\subsubsection{Diffusion up to time $s \leq T - 10$}
Given a sufficiently regular diffusion of diffusion coefficient $a/2$ and drift term $b$, one can use either Doob's H-transform \cite{RW} or direct calculations as in Pinsky \cite{PINSKY} to show that, conditioned on $\tau \in (T,T+c')$, up to time $T$ we have a non-homogeneous diffusion with the following generator
\[L_s^T = 1/2 \nabla \cdot a\nabla + b\nabla + a \frac{\nabla \PP_x(c + T -s \geq \tau > T - s)}{\PP_x(c + T-s\geq \tau > T-s)}\nabla\]  
It is also known by same methods that conditioned on everlasting survival, the generator becomes
\[L_s^\infty = 1/2 \nabla \cdot a\nabla + b\nabla + a \frac{\nabla \phi_0(x)}{\phi_0(x)}\nabla\]
In our concrete setting this means that conditioned on everlasting survival our diffusion process is given by
\begin{equation}\label{eqSur}
d\alpha^\infty_s = \sqrt{\kappa}dB_s + 2\cot \frac{\alpha^\infty_s}{2}ds
\end{equation}  
Our aim is to then show that for $T$ large at least until some time $T-10$ the diffusion conditioned to survive up to time $T$ is almost the same. More explicitly, we claim that
\begin{lemma}\label{mainterm}
The conditioned diffusion can be written as:
\begin{equation}\label{eqC}
d\alpha^T_s = \sqrt{\kappa}dB_s + (2\cot \frac{\alpha^T_s}{2} + E^T(\alpha^T_s, s))ds
\end{equation}
for some independent Brownian motion $B_t$ and the error term 
\[E^T(x,s) = \frac{\nabla \PP_x(c + T -s \geq \tau > T - s)}{\PP_x(c + T -s \geq \tau > T - s)} - \frac{\nabla \phi_0(x)}{\phi_0(x)}\] 
satisfies $E^T(x,s) \lesssim e^{-a(T-s)}$ for $s \in [0,T-10]$, for some $a > 0$ and uniformly over the interval $[0,2\pi]$. Hence in the time interval $[0, T-10]$ the conditioned diffusion is absolutely continuous with respect to the everlasting survival process.
\end{lemma}
The proof of this lemma just makes use of our control on the eigenfunctions:
\begin{proof}
We start by writing out a series representation for $\PP_x(c + T -s \geq \tau > T - s)$. To do this, notice first
that 
\[\PP_x(c + T - s \geq \tau > T-s) = \PP_x(\tau > T - s) - \PP_x(\tau > c+ T-s)\]
and so it suffices to find series representation for the similar terms on the RHS.

Now, using lemma \ref{Cl1} and the condition on the growth of eigenvalues \eqref{EVG}, it is easy to see \cite{BASS}, that for any $t > 0$ the transition probabilities of the initial process \eqref{eqD} can be written as a sum converging absolutely and uniformly over the whole interval $[0,2\pi]$: 
\[\PP_x(\alpha_t \in dy) = \sum_{i=0}\phi_i(x)e^{-\lambda_i T}\phi_i(y)m(dy)\]
Thus survival probability can be written as a series
\begin{equation}
\PP_x(\tau > T) = \sum_{i=0}c_i\phi_i(x)e^{-\lambda_i T}
\label{SP}
\end{equation}
Similarly the convergence of this sum is also absolute and uniform over the interval. Moreover, if we choose some $t_0 > 0$, then for all $T > t_0$ the convergence is uniform in $t$ as well. Any $t_0 > 0$ would do, so we pick $t_0 = 10$. 

Notice that then we can in fact write that $\PP_x(\tau > T) \asymp e^{-\lambda_0T}\phi_0$ for all $T > t_0$. This gives us in a slightly more direct manner the conclusion of the first moment argument for the Hausdorff dimension of SLE curves in \cite{BEF2}. More precisely, it replaces the hands-on technical section 1.2 of that paper by the more general setup presented here.  

Now plugging in the expansion \eqref{SP} using the remark above, we have
\[e^T(x,s) = \frac{\sum_{i=1}c_i'\left(\phi_i'(x)\phi_0(x)-\phi_0'(x)\phi_i(x)\right)e^{-\lambda_i (T-s)}}{\phi_0(x)\left(\PP_x(\tau > T-s)-\PP_x(\tau > c+ T-s)\right)}\]
with $c_i' = c_i(1-e^{-\lambda_i c})$.

We start from the denominator. Using the uniform convergence for $T-s > 10$ and $\lambda_1 > \lambda_0$ we have
\[\PP_x(\tau > T - s) - \PP_x(\tau > c+ T-s) = c_0'\phi_0(x)e^{-\lambda_0(T-s)} + O(\phi_0(x))e^{-0.5(\lambda_0+\lambda_1)(T-s)}\]
Thus we have a lower bound: 
\[\PP_x(\tau > T - s) - \PP_x(\tau > c + T-s)  \gtrsim \phi_0(x)e^{-\lambda_0(T-s)}\] 

For the nominator, write 
\[\phi_i'(x)\phi_0(x)-\phi_0'(x)\phi_i(x) = \phi_0^2(x)\left(\frac{\phi_i(x)}{\phi_0(x)}\right)'\]
Plugging in the derivative estimates from lemma \ref{Cl2} and using the bound on the growth of eigenvalues \eqref{EVG}, we have for $T-s > 10$ uniformly
\[\abs{\sum_{i=1}c_i'\left(\phi_i'(x)\phi_0(x)-\phi_0'(x)\phi_i(x)\right)e^{-\lambda_i (T-s)}} \lesssim  e^{-0.5(\lambda_1+\lambda_0)(T-s)}\phi_0^2(x)\]
And thus for $T-s > 10$ uniformly over time and space
\[\abs{E^T(x,s)} \lesssim e^{-0.5(\lambda_1-\lambda_0)(T-s)}\]
and the lemma follows.
\end{proof}

Putting things together we find the total winding of this part:
\[\int_0^{T-10} \cot \frac{\alpha_s}{2} \overset{d}{\sim}\frac{\sqrt{\kappa}}{2}B_{T-10} + (\alpha^T_{T-10} - \alpha^T_0) + \int_0^{T-10}E(\alpha^T_s,s)ds\]
Now, $\alpha^T_s$ itself is bounded and due to the exponential decay of the error term, the final term is also uniformly bounded. Finally, from the Brownian part we get a Gaussian of variance $T-10$. This gives us that conditioned on $\tau \in [T,T+c]$ we have
\begin{equation}\label{R1}
\int_0^{T-10} \cot \frac{\alpha_s}{2}ds \overset{d}{\sim}\frac{\kappa}{2}X + E_B
\end{equation}
with $X$ Gaussian of variance $T-10$ and $E_B$ some uniformly bounded random error (not independent of $X$). Looking at the exponential moments, we account for the main term of the theorem and a multiplicative error.

\subsubsection{The remaining part: $T-10 < t \leq \tau$}
Now after the time $T-10$, our control on the drift term gets gradually worse and worse and hence our previous strategy doesn't allow the exact estimation of the contribution to winding by relating it to the Brownian motion. This is due to the fact that the initial strong boundary repulsion at time 0 changes gradually to an attraction at time $T$. Hence we need a different strategy. 

We start by reducing our workload considerably: 
\begin{claim}
\label{PN}
It is sufficient to only deal with the upper bound of the exponential moments for $\lambda > 0$. 
\end{claim}

\begin{proof}
Indeed, firstly, it is easy to see that uniform upper (lower) bounds on exponential moments for $\lambda > 0$ give also lower (upper) bounds for $\lambda < 0$. 

Secondly, notice that the processes starting from $a$ and $2\pi -a$ are symmetric with respect to $\pi$, but $\cot\frac{x}{2}$ is antisymmetric. Hence we can couple processes $\alpha_1$ and $\alpha_2$ starting from $a$ and $2\pi - a$ by using the Brownian motion $B_t$ and $-B_t$ such that $\cot \frac{\alpha_1(s)}{2} + \cot \frac{\alpha_2(s)}{2} = 0$.

Hence an uniform lower bound on the positive exponential moments of $\int \cot\frac{x}{2}$ starting from $2\pi$, is via Cauchy-Schwarz equivalent to an uniform upper bound on the exponential moments and vice versa. Indeed, we can write 
\[1 = \EE\exp\left(\lambda \int_{T-10}^{\tau} \cot \frac{\alpha_1(s)}{2} + \cot \frac{\alpha_2(s)} ds\right)\]
and then Cauchy Schwarz to get
\[1 \leq \left[\EE\exp\left(2\lambda \int_{T-10}^{\tau} \cot \frac{\alpha_1(s)}{2} ds\right)\right]^{1/2}\left[\EE\exp\left(2\lambda \int_{T-10}^{\tau} \cot \frac{\alpha_2(s)}{2} ds\right)\right]^{1/2}\]

 Thus the claim follows.
\end{proof}

Now we have to treat separately cases $\kappa \neq 4$ and $\kappa = 4$. For the former, we will first discuss how to obtain a bound on the exponential moments from the time $T-9$ onwards, then deal with the middle part, i.e. the time interval $[T-10,T-9]$, and finally put them together to obtain control over the whole remaining part. Thereafter we handle the case $\kappa = 4$ in a more direct manner.

\paragraph{Control over the interval $[T-9, \tau]$ for $\kappa \neq 4$}\mbox{}\\
Suppose that at time $T-9$ we are at some point $\delta > 0$. Then the process onwards is given by the initial diffusion conditioned to die between $9 < \tau \leq 9+c$. Recall that the initial diffusion equation \eqref{eqD} has a unique strong solution and so we can work with respect to the filtration of the corresponding Brownian motion $B_t$.

Consider the exponential martingale $\exp(\lambda B_t - \lambda^2t/2)$ and the bounded stopping time $\tau' = (9+c) \wedge \tau$. We can use the optional stopping theorem to get $\EE(\exp(\lambda B_{\tau'} - \lambda^2\tau'/2)) = 1$. But on the other hand, we know that as $\alpha_s$ remains always bounded, then from the initial diffusion equation \ref{eqD} it follows that $\int_0^{\tau'} \cot \frac{\alpha_s}{2}ds = \frac{2\sqrt{\kappa}}{\kappa -4}B_{\tau'} + C'$ with $C'$ random, but in $[0, 2\pi]$.
Thus we have
\[\EE\exp(\lambda \int_0^{\tau'} \cot \frac{\alpha_s}{2}ds) \lesssim\EE\exp(\frac{4\kappa}{(\kappa-4)^2}\lambda^2\tau'/2) \lesssim \EE\exp(\frac{4\kappa}{(\kappa-4)^2}\lambda^2(9+c)/2)\] 
where the implied constants depend on $\lambda, \kappa$. Hence for any event $F$ 
\[\EE\exp(\lambda \int_0^{\tau'} \cot \frac{\alpha_s}{2}ds)|F)\PP(F) \lesssim \EE\exp(\frac{4\kappa}{(\kappa-4)^2}\lambda^2(9+c)/2)\]
In particular, we can choose the event $F = \{9 < \tau \leq 9 +c\}$. Recall from the proof of lemma \ref{mainterm} that this is of order $O(\delta^{8/\kappa - 1})$. And thus forgetting the dependence on fixed $\lambda, c,\kappa$ we get an upper bound of order $O(\delta^{1-8/\kappa})$ on the exponential moments. This of course in case we should be at $\delta$ at time $T-9$. Notice that this way we get an uniform bound for any $\delta > \delta_0$. Unfortunately this blows up as $\delta_0 \downarrow 0$.

However, if we were able to well control the probability of being below $\delta_0$ at time $T-9$ independently of the position at time $T-10$, we would stand some hope. This is indeed our plan. As is clear from the proof of lemma \ref{mainterm}, absolute continuity with respect to everlasting survival process \eqref{eqSur} lasts nicely also up to time $T-9$ (with a slightly worse constant). Following \cite{BASS} and \cite{PINSKY}, the transition probabilities for this everlasting survival process are given by $\PP_x(\alpha_t^\infty \in dy) \lesssim \sin^{\frac{8}{\kappa}}\frac{y}{2}dy$ for any $t > 0$ and thus surely at $t = 1$. 

Let us ask whether this is enough sufficient to get a nice bound: first, from absolute continuity, our conditioned process will have probability $O(\delta^{8m/\kappa+1})$ to be in the interval $[\delta^{m+1},\delta^m]$ at time $10$; second, taking the above expectation over all possible intervals of this form, we get a geometric sum of terms $O(\delta^{8m/\kappa+1}\delta^{m(1-8/\kappa)})=O(\delta^{m+1})$ which has a finite value. Thus everything looks nice. When we put things together in the end of the subsection it is cleaner to condition on the exact position of the diffusion at time $T-9$, but this just replaces sums by integrals and everything remains nicely bounded.

\paragraph{Control over the interval $[T-10, T-9]$ for $\kappa \neq 4$}\mbox{}\\
Now we deal with the small remaining part from $T-10$ to $T-9$. Again, as over this time window the process is absolutely continuous with respect to the process conditioned on everlasting survival given by \eqref{eqSur}, it is sufficient to bound exponential moments for the latter. 

It might seem that we also have an additional conditioning pushing the endpoints to lie in an interval $[\delta^{m+1},\delta^m]$. However, in fact when putting the remaining part together in the next paragraph, we will get rid of this dependence. Hence we need to just control the exponential moments independently of the starting point at $T-10$ for the process that is conditioned on the everlasting survival. Now as $\cot\frac{x}{2}$ is decreasing in $[0, 2\pi]$, then from stochastic coupling of different trajectories using the same Brownian motions, one can see that the exponential moments $\EE\exp(\lambda \int_{T-10}^{T-9} \cot \frac{\alpha_s}{2}ds)$ are bounded by those coming from the process that starts at the point $0$. 

Finally, recall the form of the everlasting survival process \eqref{eqSur}:
\begin{equation*}
d\alpha^\infty_s = \sqrt{\kappa}dB_s + 2\cot \frac{\alpha^\infty_s}{2}ds
\end{equation*}
It follows that we can write the exponential moments of $\int_0^11  \cot \frac{\alpha^\infty_s}{2}$ as above using the Brownian part:
\[\int_0^1 \cot \frac{\alpha^\infty_s}{2}ds = \frac{\sqrt{\kappa}}{2}B_1 + C'\]
with $C'$ random, but in $[0, 2\pi]$ and conclude that the exponential moments are finite, independent of where the process is at the time $T-10$. 

\paragraph{Putting the remaining part together for $\kappa \neq 4$}\mbox{}\\
Recall that the main part from the winding came from the time interval $I_1=[0,T-10]$. Additional error terms come from intervals $I_2 = [T-10,T-9]$ and $I_3 = [T-9,\tau]$. As the winding is given as an integral over time, we can decompose the winding over the remaining part $R = I_2 \cup I_3$ as 
$w_R = w_{I_2} + w_{I_3}$. Denoting by $\mathcal{F}_{I_1}$ the filtration of the underlying Brownian Motion up to to time $T-10$ , we can write the contribution of the remaining part as:
\[\EE(e^{\lambda (w_{I_2} + w_{I_3})}|\mathcal{F}_{I_1})\]
For now this is a random variable. We Cauchy-Schwarz the expectation to get rid of the dependence at the point $T-9$ and gain an upper bound
\[\EE(e^{\lambda (w_{I_2} + w_{I_3})}|\mathcal{F}_{I_1}) \leq \EE(e^{2\lambda w_{I_2}}|\mathcal{F}_{I_1})^{1/2}\EE(e^{2\lambda w_{I_3}}|\mathcal{F}_{I_1})^{1/2}\]

Now, start from the first term. As the conditioned process is a nice Markov process, what happens over the time interval $I_2 = [T-10, T-9]$ depends on the filtration $\mathcal{F}_{I_1}$ only through its position at the time $T-10$. But we saw that the positive exponential moments over $I_2$ have uniform bounds independent of the location of the process at time $T-10$. Thus: 
\[\EE(e^{\lambda (w_{I_2} + w_{I_3})}|\mathcal{F}_{I_1}) \lesssim \EE(e^{2\lambda w_{I_3}}|\mathcal{F}_{I_1})^{1/2}\]
For the second term, we condition further on the value of $\alpha_{T-9}$:
\[\EE(e^{2\lambda w_{I_3}}|\mathcal{F}_{I_1}) = \EE(\EE(e^{2\lambda w_{I_3}}|\alpha_{T-9})|\mathcal{F}_{I_1})\] 
In the discussion above we saw that 
\[\EE(e^{2\lambda w_{I_3}}|\alpha_{9}) \lesssim \alpha_{T-9}^{1-8/\kappa}\]
Thus 
\[\EE(e^{2\lambda w_{I_3}}|\mathcal{F}_{I_1}) \lesssim \EE(\alpha_{T-9}^{1-8/\kappa}|\mathcal{F}_{I_1})\]
Also, as argued above, the density of $\alpha_{T-9}$ satisfies $\PP_x(\alpha_{T-9} \in dy) \lesssim \sin^{\frac{8}{\kappa}}\frac{y}{2}dy$ independently of the starting point at $T-10$. Thus the expectation is nicely finite and indeed, putting everything together
\[\EE(e^{\lambda w_R}|\mathcal{F}_{I_1}) = O(1)\] 
where now the implied constant is deterministic. 

\paragraph{Remaining part for $\kappa = 4$}\mbox{}\\
Although the above strategy fails for $\kappa = 4$, the diffusion itself is simpler: the drift term in \eqref{eqD} vanishes and the unconditioned process is really just twice a standard Brownian motion. As we are just aiming for bounds of exponential moments, we can well assume that we have the standard Brownian motion, denote it by $B_t$.

As above we aim to find upper bounds for positive ($\lambda > 0$) exponential moments:
 
\[\EE\exp(\lambda \int_0^\tau \cot \frac{B_s}{2}ds)|\tau \in [10,10+c)\]

Start by noticing that in the space interval $[0, 2\pi]$ we have $\cot{\frac{x}{2}} \leq \frac{4}{x}$. Thus it suffices to bound

\[\EE\exp(\lambda \int_0^\tau \frac{1}{B_s}ds)|\tau \in [10,10+c)\]

Next we separate cases $B_\tau = 0$ and $B_\tau = 2\pi$. The latter case is simple, as conditioned on $B_\tau = 2\pi$, we have a Bessel-3 process. With positive probability this process reaches $2\pi$ in the time interval $[10,10+c)$. Thus it suffices to bound just the relevant exponential moments for a Bessel-3 process starting from a point in $[0,2\pi]$. This we can again do by studying the relevant SDE as above for case $\kappa \neq 4$. The SDE of Bessel-3 is given by

\[d\rho_t = dB_t + \frac{1}{\rho_t}dt\]

Writing $\tau' = \tau \wedge 10+c$, we have

\[\EE\exp(\lambda \int_0^{\tau'} \frac{1}{\rho_s}ds) \lesssim \EE e^{\lambda\rho_{\tau'}}\]

Thus, as the exponential moments for Bessel processes on the LHS certainly exist \cite{revuzyor}, we have the desired upper bound.

For the case $B_\tau = 0$ we need a bit more. Here, the idea is to condition on the exact values of exit times $\tau \in [10, 10+c)$ to obtain a family of Brownian excursions of fixed length and to gain control over these excursions. In other words, we want to write

\begin{equation}\label{kappa4}
\EE\left[\exp(\lambda \int_0^\tau \frac{1}{B_s}ds)|\tau \in [10,10+c), B_\tau = 0\right] = \EE\left[\II_{\tau \in [10,10+c)}\EE(\exp(\lambda \int_0^\tau \frac{1}{B_s}ds)|\tau,B_\tau = 0)\right]
\end{equation}

and study $\EE(\exp(\lambda \int_0^\tau \frac{1}{B_s}ds)|\tau,B_\tau = 0)$.

First, notice that by stochastic coupling using the same Brownian motion, we can certainly consider the starting point also to be at $0$. 
How to describe this conditioned process? We are conditioning on two events: 1) the process being back at zero at $\tau$ and 2) remaining inside the interval for $0 < t < \tau$. Now, as is well known, the Brownian excursion measure can defined as a limit of nicely defined conditional measures. Also, the second event has positive probability in all of the considered measures. Thus we can condition in any order. In particular we can obtain our conditioned process by taking a Brownian excursion and conditioning it to be lower than $2\pi$. Now this latter conditioning has positive probability, and so proving an upper bound on the exponential moments over the usual excursions suffices our needs. 

To control the integral over the Brownian excursion over time $[0,1]$, recall that the scaled Brownian excursion is in fact just a Bessel-3 bridge with the following SDE \cite{revuzyor}:

\[d\rho_t = dB_t + \frac{1}{\rho_t}dt - \frac{\rho_t}{1-t}dt\]

Then, as above we can write 

\[\int_0^{0.5} \frac{1}{\rho_t}dt = B_t + \int_0^{0.5}\frac{\rho_t}{1-t}dt\]

Thus denoting by $M^*$ the maximum of the Bessel bridge in $[0,1]$, we have for some positive constant $c$:

\[\EE\exp(\lambda \int_0^{0.5} \frac{1}{\rho_s}ds) \leq e^{\lambda^2/8}\EE e^{\lambda c M^*}\]

But this maximum of the Bessel 3-bridge is below the maximum of the usual Bessel 3-process in $[0,1]$, and for the latter all exponential moments exist \cite{revuzyor}. Thus $\EE\exp(\lambda \int_0^{0.5} \frac{1}{\rho_s}ds) = O(1)$. As the bridge is symmetric, it also follows that: $\EE\exp(\lambda \int_{0.5}^1 \frac{1}{\rho_s}ds) = O(1)$

Finally, by Cauchy-Schwarz we have 

\[\EE\exp(\lambda \int_0^1 \frac{1}{\rho_s}ds) = O(1)\]

Hence we have showed the existence on the relevant exponential moments over the Bessel-3 bridges of length $1$. But by scaling this amounts to the existence of these moments for all bridges of fixed lengths in $[10, 10 +c]$. Moreover these bounds are all dominated by those of the longest bridge. Thus we can uniformly uper bound the term $\EE(exp(\lambda \int_0^\tau \frac{1}{B_s}ds)|\tau,B_\tau = 0)$m in \eqref{kappa4} and obtain also $O(1)$  error bound for $\kappa = 4$ uniformly over the starting point of the error interval.

\paragraph{Negative exponential moments and lower bounds}\mbox{}\\
Finally, recall that by claim \ref{PN} in the beginning of this section, the work above for positive exponential moments also implies the upper bound for $\lambda < 0$ and lower bounds for all exponential moments. In other words we have shown that
\begin{equation}\label{R2}
 \EE(e^{\lambda w_R}|\mathcal{F}_{I_1}) \asymp 1
\end{equation}
with no randomness on the RHS. Here the implied constants depend on $\lambda, \kappa$ and can be chosen to be uniform for $\abs{\lambda} < \lambda_0$ for any choice of $\lambda_0 > 0$.

\subsubsection{The final result}
Now we individually controlled the exponential moments over time intervals $I_1 = [0,T-10]$ and $R= [T-10,\tau]$. There is one moment of dependency between them at time point $T-10$, but this does no harm as our control over the remaining part was uniform. We can write the winding as a sum over the time intervals:
\[w = w_{I_1} + w_{R}\]
Thus the exponential moments are given by 
\[\EE(e^{\lambda w}) = \EE(e^{\lambda (w_{I_1} + w_R)})\]
where we already consider the expectation with respect the common conditioning of $\tau \in [T, T+c]$. It remains then to condition out the first part:
\[\EE(e^{\lambda w}) = \EE(e^{\lambda w_{I_1}}\EE(e^{\lambda w_R}|\mathcal{F}_{I_1}))\]
where $\mathcal{F}_{I_1}$ as above denotes the filtration of the underlying BM up to the end of the first time interval.
From \eqref{R2} we know that the second term only can be adds a uniformly bounded by a deterministic constant both from above and below. Thus the proposition follows from plugging in the derived form \eqref{R1} for the first term.
\end{proof}

\begin{remark}
Of course this proof method works in a much wider context of conditioned diffusions, hence we hope it could be of some independent interest as well. 
\end{remark}

\section{Expected quantum Minkowski dimension of the SLE$_\kappa$ flow lines}

In this section we aim to find the exact expected quantum Minkowski dimension of the SLE$_\kappa$ flow lines and show that this does not satisfy the KPZ relation and to deduce that the almost sure Hausdorff version of the KPZ relation is not satisfied either. For technical reasons we now consider the unit disc as our underlying domain. 

The main result can be then stated as follows:

\begin{theorem}\label{thmD}
Consider the Liouville measure with $\gamma < 2$ in the unit disc and let $0 < \kappa < 8$. Then the expected quantum Minkowski dimension of the SLE$_\kappa$ flow lines is given by $q_{M,E} < 1$ satisfying
\[d_M = (2+\gamma^2/2)q_{M,E} - \gamma^2(1-\kappa/4)^2q_{M,E}^2/2\]
where $d_M$ is the Minkowski dimension of the respective SLE curve. 

\end{theorem}
Hence for $0 < \kappa < 8$ the KPZ relation is not satisfied for the expected Minkowski dimension. And from proposition \ref{MinHff}, we straight away deduce that:

\begin{corollary}\label{noKPZ}
Consider the Liouville measure with $\gamma < 2$ in the unit disc and let $0 < \kappa < 8$. Then almost surely the quantum Hausdorff dimension for the flow lines SLE$_\kappa$ is below the dimension predicted by KPZ relation and hence the KPZ relation is not satisfied in the almost sure Hausdorff version. 
\end{corollary}

The intuition behind this result can be gained by comparing the two images on figure \ref{fig1} in the introduction of the paper that illustrate the SLE$_{8/3}$ flow line and level line couplings. Indeed, we saw that zero level lines acted like the boundary of the domain and hence the KPZ relation was not satisfied as the field was considerably lower around them. Now looking at figure \ref{fig1} we can also see that at least for $\kappa$ close to 4, the SLE$_\kappa$ flow lines still stick close to the level line. Hence similarly to the zero level line case, the corresponding quantum contents of the coverings should be smaller and thus the quantum dimension lower.

For $\kappa = 0, \kappa = 8$ we regain the KPZ relation, which is nice but not surprising as $\kappa = 0$ should correspond to a straight line joining zero and infinity, i.e. become independent of the field, and for $\kappa = 8$ the winding part itself should form the whole field. So in some sense their behaviour is "field-independent". Here we provide two illustrative images by Scott Sheffield that indicate what happens when $\kappa$ is near $0$ or $8$:

\begin{figure}[h]
\label{fig2}
\captionsetup{width=0.8\textwidth}
\centering
\begin{minipage}[b]{0.4\textwidth}
\centering
\includegraphics[width=\textwidth]{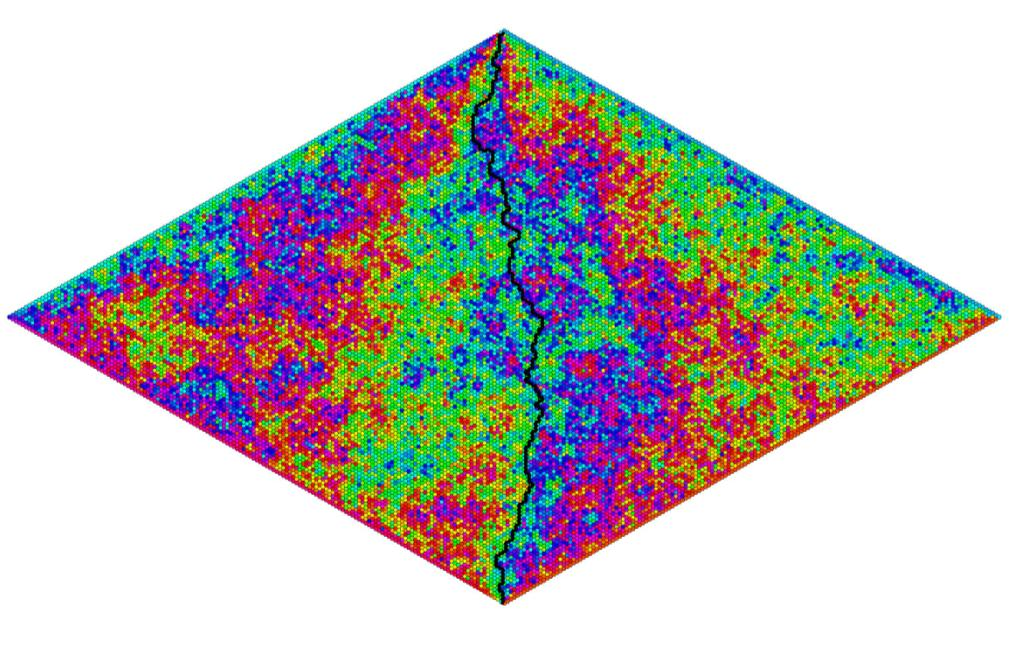}
\end{minipage}
\hspace{0.5cm}
\begin{minipage}[b]{0.4\textwidth}
\centering
\includegraphics[width=\textwidth]{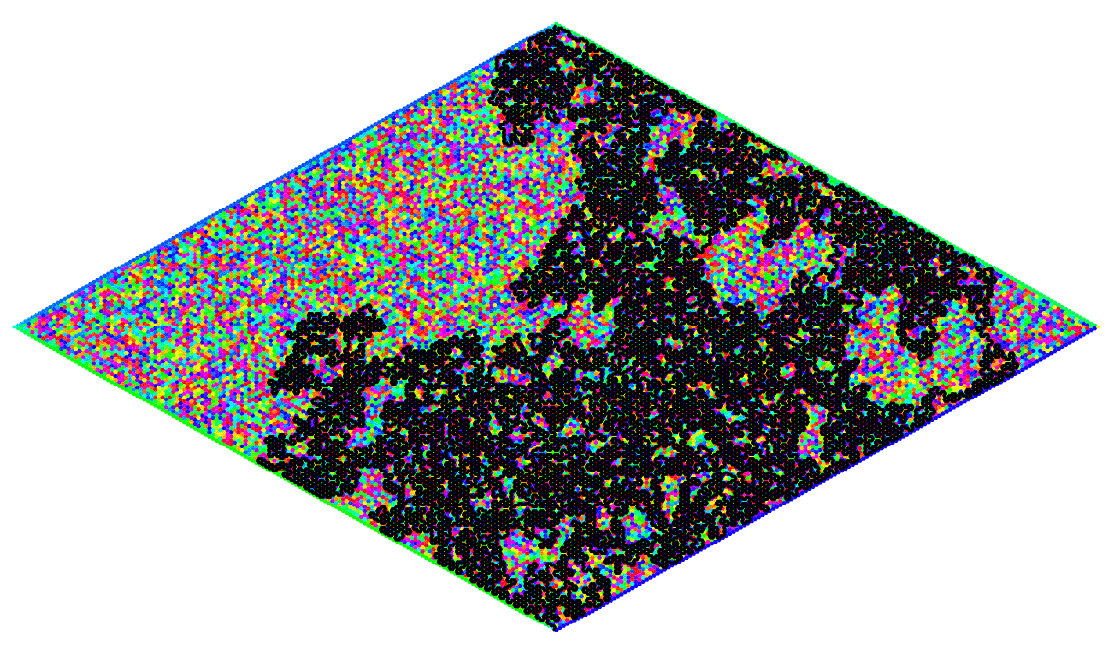}
\end{minipage}
\caption{{\small On the left the flow line corresponding to SLE$_{0.5}$ is represented. Notice that it does not really hold close to the level line anymore, but shoots quite straight from one end-point to the other. On the right we have the SLE$_{7.5}$ flow line. One can see that it starts filling the space, not being too picky about which points to step on.}} 
\end{figure}

\paragraph{Proof strategy}\mbox{}\\
Recall our simple proof strategy for $SLE_4$: cover the curve with balls, look at their scaling using Jensen to bring expectation inside integrals, and conclude. This does not seem to work here. Of course already the fact we also want lower bounds asks for some additional ideas. However, main problems are related to the additional winding term in the coupling theorem \ref{thmC} for the flow lines:

\begin{itemize}
\item First, it is crucial to take averages here over the SLE process to make use of the winding theorem \ref{thmW}. This requires us to (in some sense) fix the covering balls we are working with. Hence also the usefulness of the Minkowski version of the KPZ relation.
\item Second, the fact that winding is not defined on the SLE curve and that we can only calculate it for a specific conditioning poses its constraints.
\item Third, as a minor modification we now need to work with the chordal SLE drawn up to the very end. the underlying domain is then cut into two pieces and it needs some extra care.
\end{itemize}

Our strategy of attack makes use of a variant of the dyadic Whitney decomposition which we call conformal-radius or CR-Whitney decomposition. It allows us at the same time to work off the curve, nicely incorporate the results on winding and still get the necessary information on the fractal geometry of the curve. Whitney decomposition has been also used to study the geometry of the SLE in the beautiful seminal paper by Rohde \& Schramm on basic properties of the SLE \cite{ROHDESCHRAMM}, in particular they used it to provide the correct upper bounds for the Minkowski dimension and thus Hausdorff dimension of traces for the SLE curves.

By using the CR-Whitney decomposition, the proofs of both the upper and lower bounds for the quantum expected Minkowski dimension will follow the same outline. To bound the Minkowski dimension we need to provide bounds for the Liouville measure of a dyadic covering. We will do this in three steps: first, we estimate the expected Liouville measure of a single CR-Whitney square for the SLE slit domain; second, we provide an estimate on the expected Liouville measure over a collection of suitable CR-Whitney squares; and finally, we translate this estimate into an estimate about the combined measure of a dyadic covering.

\subsection{CR-Whitney decomposition}

Recall that dyadic Whitney decomposition of a domain is composed of dyadic squares $Q$ that satisfy: $l(Q) \leq d(Q) \leq 4l(Q)$ where $d(Q)$ is the distance of the square from the boundary of the domain and $l(Q)$ the side-length of the square.One way to achieve a dyadic Whitney decomposition is to just pick all maximal dyadic squares with $d(Q) \geq l(Q)$. The maximality will guarantee the other inequality. See for example \cite{gm} or \cite{ROHDESCHRAMM} for an usage in context. 

It comes however out that it is easier for us not to work with the usual Whitney squares, as this would make incorporating information on winding rather technical. We hence work with a slight modification, where instead of normal distance we use the conformal radius. Thus we define CR-Whitney squares as dyadic squares $Q$ such that they satisfy $4l(Q) \leq \CR(z_0) \leq 12l(Q)$. Notice that here we really condition on the conformal radius of the centre, thus allowing to use the results on winding, i.e. theorem \ref{thmW}. We have an analogous CR-Whitney decomposition, which we state for clarity as a separate lemma.

\begin{lemma}[CR-Whitney decomposition]\label{lemCR}
For every Jordain-domain of the complex plane, we can find a decomposition of dyadic squares such that any $Q \in \mathbb{W}$ satisfies $4l(Q) \leq \CR(z_0) \leq 12l(Q)$, where $\CR(z_0)$ is the conformal radius of the centre of $z_0$ of $Q$, and that the interiors of the squares do not overlap.
\end{lemma}
\begin{proof}
Again, pick all maximal dyadic squares satisfying $4l(Q) \leq CR(z_0)$. Then using the triangle inequality and the relation $CR(z_0)/4 \leq d(z_0,\partial D) \leq CR(c_0)$, we arrive that the maximality imposes $CR(z_0) \leq l(8 + 2\sqrt{2}) \leq 12l$.
\end{proof}

It is important for us that we can fully cover the slit domain with CR-Whitney squares. However, we do not actually want to further use the disjointness condition. We would like the event \{$Q$ is a CR-Whitney square\} to be in exact correspondence with conditioning on the conformal radius of its centre and sticking to the disjointness condition would ruin this. 

Hence we stress that from now on, being a CR-Whitney square only means conditioning on its centre to satisfy certain inequalities. 

\subsubsection{An estimate on the Green's function}

To estimate the Liouville measure of a CR-Whitney square, we need tight control on the Green's function inside a CR-Whitney square. This is established in the following lemma, which might be well-known, but we could not locate a concrete reference in the literature. It is similar to Harnack type of inequalities, only that we ask for additive bounds. We state and prove it first for typical Whitney squares.
\begin{lemma}\label{lemG}
Let $D$ be some bounded simply connected domain. Write the Green's function in $D$ in the form $G_D(x,y) = \log \frac{1}{\abs{x-y}} + \tilde{G}_D(x,y)$. Then in any Whitney square with $l(Q) < 1$ of $D$, we have 
\[-\log \frac{1}{d(Q,\partial D)} - C_1\leq \tilde{G}_D(x,y) \leq -\log \frac{1}{d(Q,\partial D)} + C_2\]
for some universal constants $C_1,C_2$.
\end{lemma}
However in fact we make use of the following straightforward corollary:
\begin{corollary}
The same holds for CR-Whitney squares with possibly different constants
\end{corollary}
This indeed follows quickly, as one can for example notice that any CR-Whitney square is either contained in a at most M-times bigger Whitney square or is tiled into at most M-times smaller Whitney squares for some absolute constant M. 
The proof of the lemma itself needs a bit more:
\begin{proof}[Proof of lemma \ref{lemG}]
The left-hand side is simple. For fixed $x$, $\tilde{G}_D(x,y)$ is by definition the harmonic extension to $D$ of $-\log \frac{1}{x-y}$ on $\partial D$. Now we know that a harmonic function inside a bounded domain achieves its minimum on the boundary. Combining this with the fact that the boundary of $D$ is at least at distance $d(Q,\partial D)$ for any $x,y \in Q$, we get the lower bound.

For the right-hand side, i.e. the upper bound, a lengthier argument seems to be needed. We know that the Green's function in the upper half plane is given by 
\[G_{\mathbb{H}}(z,w) = \log \frac{1}{\abs{z-w}} - \log \frac{1}{\abs{z-\overline{w}}}\]
Now pick $f: \HH \rightarrow D$ to be a conformal map and set $x = f(z)$, $y = f(w)$. Then by the conformal invariance of the Green's function, we have
\[\log \frac{1}{\abs{z-w}} - \log \frac{1}{\abs{z-\overline{w}}} = \log \frac{1}{x-y} + \tilde{G}_D(x,y)\]
Now using the complex version of the Mean Value Theorem, write $x - y = f(z) - f(w) = A(z-w)$ where $A = \operatorname{Re}(f'(u)) + i \operatorname{Im}(f'(v))$ for some $u,v$ on the line between $z$ and $w$. Plugging this into the previous equation, we get
\[\tilde{G}_D(x,y) = -\log \frac{1}{\abs{z-\overline{w}}} - \log \frac{1}{\abs{A}}\]
Now using triangle inequality, we have $\abs{z-\overline{w}} \leq \abs{z-w} + 2\operatorname{Im}(w)$.
So using also the definition of $A$ again,
\[\tilde{G}_D(x,y) \leq - \log \frac{\abs{A}}{\abs{x-y}+2\abs{A}\operatorname{Im}(w)} - \log \frac{1}{\abs{A}} = -\log \frac{1}{\abs{x-y} + 2\abs{A}\operatorname{Im}(w)}\]
Now we know that $\abs{x-y} \leq \sqrt{2}l(Q)$. Also, we know that for Whitney squares the side-length is up to fixed multiplicative constants equal to the distance of the boundary. Thus $\abs{x-y} \leq cd(Q,\partial D)$.

Recall that from distortion theorems \cite{POMM} it follows that for $f$ analytic from $D_1 \rightarrow D_2$ we have
\begin{equation}
\label{dist}
\abs{f'(z_0)} \asymp \frac{d(f(z_0),D_2)}{d(z_0,D_1)}
\end{equation}
where the implied constants are absolute. 
Thus we get that 
\[d(Q) \lesssim \operatorname{Im}(w)\abs{f'(w)} \lesssim d(Q))\]
for some absolute constants and hence
\[\tilde{G}_D(x,y) \leq -\log \frac{1}{cd(Q,\partial D)} - \log \frac{\abs{f'(w)}}{\abs{A}} + C\]
for some absolute constant $C$. It finally remains to show an absolute bound on $\abs{A}/\abs{f'(w)}$ to conclude the lemma.

Now, we know that $Q$ can be covered by at most $M$ images of Whitney squares in $\HH$, where $M$ is a universal constant \cite{gm}. Join these $M$ Whitney squares with further Whitney squares in $\HH$ to make the region covered convex, i.e. a big rectangle. The number of these additional squares can again be universally bounded. 

Then $z,w,u,v$ lie inside this region, and as they are only bounded hyperbolic distance apart, the ratio of their imaginary parts is bounded. On the other hand this bounded number of Whitney squares can be in turn covered by a uniformly bounded number of connected Whitney squares in $D$. Thus also the ratios of distances of $f(z),f(w),f(u),f(v)$ from the boundary are bounded by constants. It follows again from the distortion theorems \eqref{dist} that also the ratios of the different $f'(\cdot)$ are bounded, giving us the claim. 
\end{proof}

\subsubsection{Controlling winding inside a CR-Whitney square}

A priori, conditioned on a dyadic square to be a CR-Whitney square we have information on its winding only at the center of the square. This could be a problem, as we have no control on the covariance structure of the winding. However, from the geometric intuition of the winding number, it is clear that inside a CR-Whitney square the winding has to be bounded up to an additive constant. Although the definition of winding in our case is different (see discussion after the statement of theorem \ref{thmW}), this result also holds in our case. Again we state and prove it for more traditional Whitney squares, but use for CR-Whitney squares:

\begin{lemma}\label{WB}
Suppose $Q$ is a Whitney square in the slit domain. Then the winding $\omega(z)$ satisfies $\omega(z) - c \leq \omega(z_0) \leq \omega(z) + c$, where $z_0$ is the centre of the square and $c > 0$ is some absolute constants.
\end{lemma}

\begin{proof}
We start by using the Borel-Carathéodory theorem \cite{TITCH}. In a slightly constrained form it states that the modulus of the analytic function $g(z)$ with $g(0) = 0$ inside a closed disc of radius $r < R$ can be controlled by the maximum of its real part on the circle of radius $R$. More explicitly, we have 
\[\abs{g(z)} \leq \frac{2r}{R-r}\sup_{z \in \partial B(0,R)} \Re g(z)\]
We want to apply this theorem with 
\begin{itemize}
\item $g(z) = \log f_T'(z)-\log f_T'(z_0)$, where $f_T$ is the map from the SLE slit domain back to the upper half plane $\HH$ and $z_0$ is the center of our Whitney square $Q$
\item $r = \frac{l(Q)}{\sqrt{2}}$ and $R = l(Q)$ with $l(Q)$ as before the sidelength of $Q$
\end{itemize}
Firstly, as our domain in question is simply connected and $f_T'(z)$ is non-zero everywhere, it follows that $g(z)$ is analytic. Secondly, the whole square $Q$ fits in the closed disc of radius $r$ and the larger disc still fits into the domain as $d(z_0,\partial H_t) \geq \frac{3l(Q)}{2}$.

Next, we need to control the real part of $g(z)$. This real part is given by 
\[\Re g(z) = \log \frac{\abs{f_T'(z)}}{\abs{f_T'(z_0)}}\]

Now it can be seen that the disc of radius $R$ centred at $z_0$ is of bounded hyperbolic diameter that is independent of the sidelength of the square $l(Q)$ and the domain. Hence by conformal invariance of the hyperbolic distance, also the images $f_T(z)$ and $f_T(z_0)$ are only at bounded hyperbolic distance. It follows from distortion theorems \eqref{dist} that the ratio $\frac{\abs{f_T'(z)}}{\abs{f_T'(z_0)}}$ is bounded by an absolute constant. Thus the same holds for $\Re g(z)$ .  

Finally, the relative change in winding w.r.t $z_0$ is given exactly by the imaginary part of $g(z)$ and the lemma follows.
\end{proof}

\begin{corollary}
The same holds for CR-Whitney squares with a slightly different constant.
\end{corollary}

\subsection{Proof of the theorem \ref{thmD}}
Now we are set to prove the theorem \ref{thmD}. We start with the upper bound and follow the strategy outlined in the beginning of the section. In all sections we start by sampling an SLE$_\kappa$ and then constructing the Liouville measure in the slit domain, using the coupling results between the GFF and SLE. We make a few remarks that simplify the further work and its write-up
\begin{enumerate}
\item We ignore at all phases the bounded harmonic correction term in the coupling, because it only gives a bounded multiplicative constant as discussed in proof of \ref{LL1}. 
\item As we sample the SLE curve until it cuts the unit disc into two, we are left with two independent GFFs in both subdomains. However we can still consider the Whitney decomposition of the unit disc with the SLE curve, and all estimates for a single Whitney square depend only one one of these GFFs, hence we can also forget about this additional issue.
due. For $\kappa = 4$ one needs to forget about winding and everything will go through. For $\kappa > 4$ one needs to notice that $\chi$ changes sign and additionally take care of sampling GFF independently in every subdomain as explained in remarks after theorem \ref{thmC}. Otherwise everything is exactly the same - indeed, even for points cut-off from infinity by the curve, the winding is defined similarly in the coupling theorem \ref{thmC} and the theorem on winding \ref{thmW}. 
\end{enumerate}

\subsubsection{Upper bound}

\paragraph{Upper bound for a CR-Whitney square}\mbox{}\\
Consider a dyadic square $Q$ of side-length $l(Q)$ and denote by $\mathcal{W}$ the collection of all CR-Whitney squares of the unit disc cut by the SLE curve. We will find an upper bound to
\[\EE_{SLE}\left[\EE_h\left( \tilde{\mu}(Q)^q\right)|Q \in \mathcal{W}\right]\]
where informally $\tilde{\mu}(dz) \asymp \mu(dz)e^{-\gamma \chi w(z)}$ is the Liouville measure now weighted by the winding. This can be given concrete meaning using the circle-average regularization process as in section 4. As winding is harmonic inside the slit domain, then taking the regularization term $\delta_n \leq 0.01l(Q)$, the circle-averages for winding give its value at the centre. Now, from the corollary to lemma \ref{WB} one can see that inside a CR-Whitney square, the winding is equal up to a constant. So setting $z_0$ to be the centre of $Q$ we can write
\[\EE_{SLE}\left[\EE_h\left( \tilde{\mu}(Q)^q\right)|Q \in \mathcal{W}\right] \asymp \EE_{SLE}\left[e^{-\gamma \chi q w(z_0)}\EE_h\left(\mu(Q)^q\right)|Q \in \mathcal{W}\right]\]
Now, with only minor modifications we can use lemma \ref{LL1}, to upper bound the Liouville part without winding and get:
\[\EE_h\left(\mu(Q)^q)\right) \leq l(Q)^{(2+\gamma^2/2)q}\]

So we are left with 
\[\EE_{SLE}\left[\EE_h\left( \tilde{\mu}(Q)^q\right)|Q \in \mathcal{W}\right] \lesssim l(Q)^{(2+\gamma^2/2)q}\EE_{SLE}\left[e^{-\gamma \chi q w(z_0)}|Q \in \mathcal{W}\right]\]
But the as $Q$ has side-length $l(Q)$ and is conditioned to be a CR-Whitney square, we are exactly conditioning the conformal radius $\CR(z_0,SLE) \in [4l(Q), 12l(Q)]$. Hence using the theorem on winding \ref{thmW}, we have 
\[\EE_{SLE}\left[e^{-\gamma \chi q w(z_0)}|Q \in \mathcal{W}\right] \lesssim l(Q)^{-\gamma^2(1-\kappa/4)^2q^2/2}\]
Putting everything together, gives us 
\[\EE_{SLE}\left[\EE_h\left( \tilde{\mu}(Q)^q\right)|Q \in \mathcal{W}\right] \lesssim l(Q)^{(2+\gamma^2/2)q - \gamma^2(1-\kappa/4)^2q^2/2}\]
\paragraph{Upper bound for Liouville measure over all CR-Whitney squares}\mbox{}\\
Next, let $\mathcal{W}_{\geq n}$ denote the collection of Whitney squares of side-length at most $2^{-n}$ we provide an upper bound for the sum
\[\EE_{SLE}\EE_h\left(\sum_{Q \in \mathcal{W}_{\geq n}} \tilde{\mu}(Q)^q\right) = \sum_{Q\in \mathcal{S}_{\geq n}}\EE_{SLE}\left[ \EE_h\left( \tilde{\mu}(Q)^q\right)|Q \in \mathcal{W}\right]\PP_{SLE}(Q \in \mathcal{W})\]
where the sum is over the collection $\mathcal{S}_{\geq n}$ of dyadic squares of side-length at most $2^{-n}$. Now for $Q$ to be a CR-Whitney square, we certainly need its center $z_0(Q)$ to satisfy $\CR(z_0) \leq 12l(Q)$. However, we know from \cite{BEF} that the probability of this happening is bounded by $O(1)l(Q)^{1-\kappa/8}$ and so
\[\PP_{SLE}(Q \in \mathcal{W}) \leq \PP_{SLE}\left[\CR(z_0) \leq 12l(Q)\right] \lesssim l(Q)^{1-\kappa/8}\]
Hence, fixing some $n \in \mathbb{N}$ as the maximal size of the dyadic squares used, and combing this previous estimate with the previous one for CR-Whitney squares, we have that for any $1 > q > 0$, $\delta > 0$ with 
\[(2+\gamma^2/2)q - \gamma^2(1-\kappa/4)^2q^2/2 = 1+\kappa/8 + \delta\] 
the following upper bound bound holds:
\[\EE_{SLE}\EE_h\left(\sum_{Q\ \in \mathcal{W}_{\geq n}} \tilde{\mu}(Q)^q\right) \lesssim \sum_{k \geq n} \sum_{l(Q) = 2^{-k}}2^{2k}2^{-k(2+\delta)} = \frac{2^{-n\delta}}{1-2^\delta}\]
Notice that by making $n$ large enough we can in fact make this sum arbitrarily small.

\paragraph{Almost sure upper bound for the covering}\mbox{}\\
The final step of the proof is inspired by the (not yet published) book of Peres \& Bishop \cite{BP}, where they discuss the notion of dimension related to Whitney decompositions. Suppose we have a covering of the SLE by dyadic squares $S_i \in \mathcal{S}_n$ such that their side-length is $2^{-n}$. The idea is to cover each dyadic squares by $CR$-Whitney squares and obtain an estimation this way for the dyadic covering. The problem is that with Whitney square we never touch the curve itself, so in order to proceed we need the following claim:

\begin{claim}\label{FLZ}
For $\kappa < 8$ the Liouville measure of SLE$_\kappa$ in forward coupling with the GFF is almost surely zero.
\end{claim}

Before proving the claim, let us show it implies the upper bound. Consider again the collection of dyadic CR-Whitney squares of side-length at most $2^{-n}$, denoted by $\mathcal{W}_{\geq n}$ and a dyadic square $S_i \in \mathcal{S}_n$ intersecting the SLE curve. Recall that the $CR$-Whitney squares cover the whole slit domain, also notice that no CR-Whitney square intersecting $S_i$ can be larger than $S_i$ itself. Hence if the Liouville measure of the curve itself is almost surely zero, we have:
\[\tilde{\mu}(S_i) \leq \sum_{Q \in \mathcal{W}_i}\tilde{\mu}(Q)\]
 where $\mathcal{W}_i$ denotes the collection of dyadic CR-Whitney squares intersecting the interior of $S_i$. 

As all CR-Whitney squares used are in fact inside $S_i$, it follows that
\[\sum_{Q \in \mathcal{W}_i}\tilde{\mu}(Q)^q \leq \sum_{\mathcal{W}_i}\tilde{\mu}(Q)\tilde{\mu}(Q)^{q-1}\]
But for $q < 1$, we have $\tilde{\mu}(Q)^{q-1} \geq \tilde{\mu}(S_i)^{q-1}$ and so
\[\sum_{Q \in \mathcal{W}_i}\tilde{\mu}(Q)^q \geq \tilde{\mu}(S_i)^q\]
Now as the collections of CR-Whitney squares $\mathcal{W}_i$ used to cover each dyadic square that intersects the SLE curve are disjoint, we have:
\[\sum_{S_i \in \mathcal{S}_n} \II(S_i \cap SLE \neq \emptyset)\tilde{\mu}(S_i)^q \leq \sum_i\sum_{Q \in \mathcal{W}_i}\tilde{\mu}(Q)^q \leq  \sum_{Q \in \mathcal{W}_{\geq n}} \tilde{\mu}(Q)^q\]
We can put everything together in expectation to get:
\[\EE_{SLE} \EE_h\left(\sum_{S_i \in \mathcal{S}_n} \II(S_i \cap SLE \neq \emptyset)\tilde{\mu}(S_i)^q\right) \leq \EE_{SLE} \EE_h\left(\sum_{Q \in \mathcal{W}_{\geq n}} \tilde{\mu}(Q)^q\right)\]
Plugging in the estimate from the last section, we obtain:
\[\EE_{SLE} \EE_h\left(M_q^Q(SLE,2^{-n})\right) \lesssim \frac{2^{-n\delta}}{1-2^\delta}\]
\[\limsup \limsup_{n \uparrow \infty} \EE_{SLE,h} \left(M^Q_q(SLE,2^{-n}) < \infty \right)\] 
Hence we see that $q_{M,E} < q$ for any $q$ such that there is a $\delta > 0$ with
\[(2+\gamma^2/2)q - \gamma^2(1-\kappa/4)^2q^2/2 = 1+\kappa/8 + \delta\] 
Now we can just let $\delta \downarrow 0$ to obtain the claimed upper bound. 

\begin{proof}[Proof of claim \ref{FLZ}]\mbox{}\\
It only remains to prove that the Liouville measure for the SLE$_\kappa$ flow lines with $\kappa < 8$ is zero. We do it using a global "no loss of mass" argument. As this involves several changes of integrals and limits, we have to be careful at all steps.

Denote by $D$ the unit disc. Pick $\epsilon \rightarrow 0$ along powers of two. Then from the definition of the Liouville measure \ref{Liouville}, we have
\begin{equation*}
\begin{split}
\mu(D) &= \lim_{\epsilon \rightarrow 0} \EE\int_{D} \mu_\epsilon(z)dz \\
& = \lim_{\epsilon \rightarrow 0}\int_{D} \EE \mu_\epsilon(z)dz\\
& = \int_{D}\lim_{\epsilon \rightarrow 0} \EE \mu_\epsilon(z)dz
\end{split}
\end{equation*}
Here, the second equality follows from Fubini and the third from dominated convergence. Now fix $m$ large and write $A_m$ for the event that the flow line avoids the $\epsilon^m$ ball around $z$, i.e. set $A_m = \{SLE \cap B_{\epsilon^m}(z) = \emptyset\}$. Then we can continue by writing
\begin{equation*}
\begin{split}
\mu(D) &= \int_{D}\lim_{\epsilon \rightarrow 0}\Big(\EE_{SLE}\EE_h(\mu_\epsilon(z)\II(A_m))+\EE_{SLE}\EE_h(\mu_\epsilon(z)\II(A_m^c))\Big)dz\\
& = \int_{D}\Big(\lim_{\epsilon \rightarrow 0} \EE_{SLE}\EE_h(\mu_\epsilon(z)\II(A_m))+\lim_{\epsilon \rightarrow 0}\EE_{SLE}\EE_h(\mu_\epsilon(z)\II(A_m^c)\Big)dz
\end{split}
\end{equation*}
By boundedness and positivity writing the limit of a sum as sum of limits is fine. We bound the second term using Cauchy-Schwarz:
\[\EE_{SLE}\EE_h\Big(\mu_\epsilon(z)\II(A_m^c)\Big) \leq \left(\EE\mu_\epsilon(z)^2\right)^{1/2}\PP(A_m^c)^{1/2}\]
But we know that $\PP(A_m^c) \asymp \epsilon^{m(1-\kappa/8)}$. By plugging in $\mu_\epsilon(z) = \epsilon^{\gamma^2/2}e^{h_\epsilon(z)}$ and using the exponential moments of Gaussians, we see that the first term is bounded by $\epsilon^{-\gamma^2/2}$. Thus the whole term is of order $O(\epsilon^{-\gamma^2/2+m/2(1-\kappa/8)})$ and by picking $m$ large enough, we can force it to be $o(\epsilon)$. But then 
\begin{equation*}
\begin{split}
\mu(D) & = \int_{D}\left(\lim_{\epsilon \rightarrow 0} \EE_{SLE}\EE_h(\mu_\epsilon(z)|A_m)\PP(A_m)+\lim_{\epsilon \rightarrow 0}\EE_{SLE}\EE_h(\mu_\epsilon(z)|A_m^c)\PP(A_m^c)\right)dz \\
&= o(\epsilon) + \int_{D}\lim_{\epsilon \rightarrow 0} \EE_{SLE}\EE_h(\mu_\epsilon(z)|A_m)\PP(A_m)dz\\
\end{split}
\end{equation*}
Here we have also integrated the error term over the domain that has bounded mass. 

Now notice that in the second term of the final expression, we never consider the mass on the curve itself. Yet there is no loss of total mass. Thus, in expectation, the mass on the curve is zero. Finally, the mass is clearly non-negative and hence it must be almost surely zero.
\end{proof}

\begin{remark}
In fact this is the claim where really the fractal geometry of the SLE, the coupling of GFF \& SLE and the construction of Liouville measure are all mixed together.
\end{remark}

\subsubsection{Lower bound}
The strategy is very similar, though small changes are needed at every step: 
\paragraph{Lower bound for a CR-Whitney square}\mbox{}\\
Again, to start off consider a dyadic square $Q$ of side-length $l(Q)$ and denote by $\mathcal{W}$ the collection of CR-Whitney squares of the unit disc cut by the SLE curve. We aim to provide a lower bound to
\[\EE_{SLE}\left[\EE_h\left( \tilde{\mu}(Q)^q\right)|Q \in \mathcal{W}\right]\]
where as before $\tilde{\mu} \asymp \mu(z)e^{-\gamma \chi w(z)}$ is informally the Liouville measure weighed by the winding. From lemma \ref{WB} we see that $w(z) \leq w(z_0) + C'$, where $z_0$ is the centre of $Q$.  So we can write
\[\EE_{SLE}\left[\EE_h\left( \tilde{\mu}(Q)^q\right)|Q \in \mathcal{W}\right] \asymp \EE_{SLE}\left[e^{-\gamma \chi q w(z_0)}\EE_h\left(\mu(Q)^q\right)|Q \in \mathcal{W}\right]\]
We need to lower bound $\EE_h\left(\mu(Q)^q\right)$
and this can be done using Kahane convexity inequalities \cite{KAHANE,RRV}, that reduce comparing the moments of balls in multiplicative chaos measures to a comparison of covariance kernels. 

To apply Kahane convexity inequalities directly, we need to change the regularization of the Liouville measure to use the exact variance, as used in the literature on the multiplicative chaos. Start by picking $\delta_n = 2^{-n}$ to get the regularization sequence for the construction of Liouville measure in theorem \ref{Liouville}. We have for $\delta_n < 0.01l(Q)$,
\[\mu_h(Q) = \lim_{\delta_n \downarrow 0} \mu_{h^{\delta_n}}(Q) = \lim_{\delta_n \downarrow 0} \int_{Q} \delta_n^{\gamma^2/2} e^{\gamma h_{\delta_n}(z)}dz \]
where $h_{\delta_n}(z)$ is a Gaussian field with the kernel
\[G_{\delta_n}(x,y) = \log \frac{1}{\delta_n \vee \abs{x-y}} + \tilde{G}(x,y)\]
Notice that as in the whole square we are at distance at least say $10\delta_n$ from the boundary, we indeed have inside our square $\tilde{G}(x,y) = \tilde{G}_{\delta_n}(x,y)$ where the former is the harmonic correction corresponding to the usual Green's function of the domain, and the latter is the harmonic correction corresponding to regularized Green's function.

Thus $\mu_h(Q)$ can be rewritten in terms of Gaussian multiplicative chaos as
\begin{equation}\label{murepr}
\mu_h(Q) \asymp l(Q)^{\gamma^2/2} \lim_{\delta_n \downarrow 0}\int_{Q} e^{\gamma h_{\delta_n}(z)-\gamma^2/2 \EE(h_{\delta_n}(z)^2)}dz
\end{equation}
We now consider two Gaussian fields $h_1$, $h_2$, with the covariance kernels respectively denoted by $G_1(x,y)$ and $G_2(x,y)$ and given as follows:
\[G_1(x,y) = G_(x,y) + \log\frac{1}{l(Q)} + C\]
for some constant $C$. Now, we take this constant from lemma \ref{lemG}. Thus when we define
\[G_2(x,y) = \log \frac{1}{\abs{x-y}}\]
Then we have that $G_2 \geq G_1$. Moreover, we can consider only sufficiently small Whitney squares such that $\log\frac{1}{l(Q)} + C$ is positive and hence $h_1$ can be written as a sum of the Gaussian free field and an independent Gaussian $Y$ of variance $\log \frac{1}{l(Q)} + C$. 
Now, by \cite{KAHANE,GMC} we know that the multiplicative chaos measures for these fields are nicely defined and we will denote them by just "$e^{h_1(z) - \EE(h_1(z)^2)}$" etc. Hence as $q < 1$ and thus $x \rightarrow x^q$ is concave, we have by Kahane convexity inequalities \cite{KAHANE,RRV}:
\[\EE \left(\int_{Q} e^{\gamma h_1(z)-\gamma^2/2 \EE(h_1(z)^2)}\right)^q \geq 
\EE \left(\int_{Q} e^{\gamma h_2(z) -\gamma^2/2 \EE(h_2(z)^2)}\right)^q\]
Using the fact that $h_1 = h + Y$, that $Y$ is an independent Gaussian and that $h, h_2$ satisfy the scaling relation \ref{scalinglemma} \cite{RRV}, we have
\[\EE \left(\int_{Q} e^{\gamma h(z)-\gamma^2/2 \EE(h(z)^2)}\right)^q \geq l(Q)^{2q}\]
Finally as $G_{\delta_n}(x,y) \leq G_(x,y)$ we can translate this back to the regularized field to get:
\[\EE \left(\int_{Q} e^{\gamma h_{\delta_n}(z)-\gamma^2/2 \EE(h_{\delta_n}(z)^2)}dz\right)^q \geq l(Q)^{2q}\]
and thus $\mu_h(Q) \lesssim l(Q)^{(2+\gamma^2/2)q}$
So taking the expectation w.r.t. SLE, we have
\[\EE_{SLE}\left[\EE_h\left( \tilde{\mu}(Q)^q\right)|Q \in \mathcal{W}\right] \gtrsim l(Q)^{(2+\gamma^2/2)q}\EE_{SLE}\left[e^{-\gamma \chi q w(z_0)}|Q \in \mathcal{W}\right]\]
But the as $Q$ has side-length $l(Q)$ and is conditioned to be a CR-Whitney square, we are conditioning on 
\[\CR(z_0,SLE) \in [4l(Q), 12l(Q)]\]
Hence using the theorem \ref{thmW}, we have 
\[\EE_{SLE}\left[e^{-\gamma \chi q w(z_0)}|Q \in \mathcal{W}\right] \gtrsim l(Q)^{-\gamma^2(1-\kappa/4)^2q^2/2}\]
Putting everything together, gives us 
\[\EE_{SLE}\left[\EE_h\left( \tilde{\mu}(Q)^q\right)|Q \in \mathcal{W}\right] \gtrsim l(Q)^{(2+\gamma^2/2)q - \gamma^2(1-\kappa/4)^2q^2/2}\]
\paragraph{Lower bound for Liouville measure over level-$n$ CR-Whitney squares}\mbox{}\\
This time we do not aim to bound the whole CR-Whitney decomposition, but are happy with analysing the collection of level-$n$ CR-Whitney squares $\mathcal{W}_n$. Moreover, we relax the definition of CR-Whitney square and call every dyadic square satisfying $4l(Q)\leq \CR(z_0) \leq 150l(Q)$ a CR-Whitney square, where as before $z_0$ is the centre of $Q$. The reason will become clear when we aim for the lower bound of the dyadic covering.

Write as earlier
\[\EE_{SLE}\EE_h\left(\sum_{\mathcal{W}_n} \tilde{\mu}(Q)^q\right) = \sum_Q\EE_{SLE}\left[\EE_h\left( \tilde{\mu}(Q)^q\right)|Q \in \mathcal{W}_n\right]\PP_{SLE}(Q \in \mathcal{W}_n)\]
and pick $1 > q > 0$, $\delta > 0$ with 
\[(2+\gamma^2/2)q - \gamma^2(1-\kappa/4)^2q^2/2 = 1+\kappa/8 - \delta\]

Now the probability of being a CR-Whitney square can be exactly calculated using the SLE Green's function \cite{LAWLERG}, and is still of order $O(1)l(Q)^{1-\kappa/8}$. Thus using this probability and the estimate on the CR-Whitney square itself we finally get
\[\EE_{SLE}\EE_h\left(\sum_{\mathcal{W}_n} \tilde{\mu}(Q)^q\right) \gtrsim  2^{2n}2^{-n(2-\delta)} \geq 2^{n\delta}\]
which is arbitrarily large for $n$ large.

\paragraph{Lower bound for the covering}\mbox{}\\
To make the final step from the lower bound on CR-Whitney squares to a lower bound on the covering, our idea is to locate at least one CR-Whitney square inside each dyadic square in the covering of the SLE. At first sight this might seem hard, because we would also need to handle the case when SLE almost fills the square. However, due to estimates of the SLE Green's function, it costs us nothing to require the SLE curve to leave some open space around the centre of the square, just enough to fill in some CR-Whitney squares.

To be more precise, notice first that in order for a dyadic square $S$ of side-length $l(S) = 2^{-n}$ to intersect the SLE curve, it suffices that the centre of this square has conformal radius less than $l(S)/2$. On the other hand we can also require the conformal radius to be more than $l(S)/3$ without changing the order of magnitude of our event \cite{LAWLERG}. 

Then a small geometrical calculation shows that all four dyadic squares of side-length $l(S)2^{-6}$ neighbouring the centre of square $S$ will necessarily be CR-Whitney square. This is of course also the reason for relaxing the definition of CR-squares in the previous section. 

The rest now follows easily. Indeed, cut $S_i$ first into four dyadic square $Q'_{i,j}$ with $j=1,2,3,4$ of sidelength $l(S_i)2^{-1}$. Then from Jensen applied to the concave function $x^q$:
\[\sum_{S_i \in \mathcal{D}_n} \II(SLE \cap S_i)\tilde{\mu}(S_i)^q \gtrsim \sum_{S_i \in \mathcal{D}_n} \II(SLE \cap S_i)\sum_{j=1,2,3,4}\tilde{\mu}(Q'_{i,j})^q\]
Now denote by $Q_{i,j}$ the corresponding dyadic squares of sidelength $l(S_i)2^{-1}$ that have the centre of $S_i$ as one corner. Thus
\[\sum_{j=1,2,3,4}\tilde{\mu}(Q'_{i,j})^q \geq \sum_{j=1,2,3,4}\tilde{\mu}(Q_{i,j})^q\]

But we saw above $\{SLE \cap S_i\} \supset \cup_{j=1,2,3,4}\{Q_{i,j} \in \mathcal{W}_{n+6}\}$ and so \[\II(SLE \cap S_i) \geq 1/4\sum_{j=1,2,3,4}\II(Q_{i,j} \in \mathcal{W}_{n+6})\]
Thus we can further lower bound the RHS by a sum over the CR-Whitney squares on level $n+6$ that are around the centre of a level $n$ dyadic square. When we denote this specific collection by $\mathcal{W'}_{n+6}$, we have:
\[\EE_{SLE}\EE_h\sum_{S_i \in \mathcal{D}_n}\II(SLE \cap S_i)\tilde{\mu}(S_i)^q \gtrsim \EE_{SLE}\EE_h\left(\sum_{\mathcal{W'}_{n+6}} \tilde{\mu}(Q)^q\right) \]
Now, $\mathcal{W'}_{n+6}$ forms a constant proportion of all CR-Whitney squares of size $n+6$, and thus we can use the previous estimate on the sum of $n$-th level Whitney squares.  Thus we get that for $n$ large enough
\[\EE_{SLE}\EE_h M_q^Q(SLE,2^{-n}) \gtrsim 2^{n\delta}\]
From this it follows that $q < q_{M,E}$ for any $q$ such that there is a $\delta > 0$ with
\[(2+\gamma^2/2)q - \gamma^2(1-\kappa/4)^2q^2/2 = 1+\kappa/8 - \delta\]
The lower bound for the expepcted quantum Minkowski dimension follows by taking $\delta \downarrow 0$. This also finishes the proof of the theorem \ref{thmD}.

\section{Further questions and speculations}

Finally, we list a few open questions to point in future directions and in hope that they could provoke some thought. We start from more realistic questions and finish in a more speculative spirit.

The first natural question is to what extent these results carry over to processes related to SLE$_\kappa$ that are also coupled to the GFF. Firstly, there are CLE processes. For example CLE$_4$ should describe the contour lines of the GFF. It is natural to expect that all the results from this paper should carry over for these CLE processes coupled with the field and in fact give the same KPZ-like relations. Indeed, as soon as exact coupling results have been published, one can probably answer the following question:
\begin{question}
Show that the KPZ type of relation for CLE loops is the same as above, and if possible, find also the KPZ type of relation for the CLE gasket.
\end{question}

Similarly, similarly to usual SLE$_\kappa$, the SLE$_{\kappa,\underline{\rho}}$ processes are also coupled with the GFF as flow lines \cite{IM}, which of course hints the following question:
\begin{question}
Determine the KPZ type of relation for all flow and counterflow lines of the GFF. 
\end{question}

What if we condition the SLE curve to pass closely through two different points, how are their windings related? It would be interesting to attack this question by trying to make use of the techniques developed for the SLE-GFF coupling, introduced in \cite{IM} and its subsequent papers. Answering the question might allow us to replace the expected quantum Minkowski dimension with an almost sure version:

\begin{question}
Find the exponential moments of $w_{z_1} + w_{z_2}$ conditioned on the SLE to pass close by points $z_1$ and $z_2$. 
\end{question}

We finish with a more general question. In this paper we showed that there is natural deviation from the KPZ relation. But by how much can one deviate? One would expect there to be non-trivial upper bounds at least for sufficiently small positive Hausdorff dimension, as then they cannot be filled with $\gamma-$thick points. Similarly one would hope for non-trivial lower bounds for at least large enough Hausdorff dimension. 
\begin{question}\label{questB}
Find best bounds $ub(\gamma,q)$ and $lb(\gamma,q)$ such that for any $0 \leq \gamma < 2$ and $0 \leq q \leq 2$ and any (possibly field-dependent) set $A$ of Hausdorff dimension $q_H(A) = q$, we have $lb(\gamma,q) \leq q^Q_H(A) \leq ub(\gamma,q)$ where the quantum dimensions are defined with respect to the $\mu_\gamma$ Liouville measure.
\end{question}

{\small
\bibliographystyle{plain}
\bibliography{biblioS}
}
\section{Appendix}

\begin{proof}[Proof of lemma \ref{lemB}]
The proof is the first moment argument in \cite{BEF}, with two slight differences: 1) we follow the evolution of the conformal radius and not the distance itself 2) we also follow the time evolution of winding. The basic strategy is the following: we transform our chordal SLE in $\HH$ to a process in $\mathbb{D}$ for which the image of $z_0$ is fixed. Then pick a convenient time change, and study the process induced for the driving Brownian motion. As in \cite{BEF} one works with the map $g_t(z)$ instead of $f_t(z)$ and we want to keep close to his exposition, we first remark that for the question of winding as defined in \ref{defW} this is equivalent - $g_t'(z)$ is equal to $f_t'(z)$.

\paragraph{Fixing the image of $z_0$}\mbox{}\\
Denote by $H_t = \HH\backslash K_t$ the SLE slit domain and consider the map $\tilde{g_t}:  H_t \rightarrow \mathbb{D}$ from the slit domain to the unit disc, given by
\[\tilde{g_t}: z \rightarrow \frac{g_t(z) - g_t(z_0)}{g_t(z) - \overline{g_t(z_0)}}\]
It maps $\infty \rightarrow 1$ and $z_0 \rightarrow 0$. We have that 
\[\log \tilde{g_t}'(z) = \log g_t'(z) - \log (g_t(z) - \overline{g_t(z_0)})\]
First of all, one can see that the conformal radius 
\[\CR(z_0,H_t) = \frac{1}{\abs{\tilde{g_t}'(z_0)}}\]
Second, we have that
\[\arg \tilde{g_t}'(z_0) = \arg g_t'(z_0) -\pi/2\]
Hence $\partial_t w(z_0) = \partial_t\arg \tilde{g_t}'(z_0)$ and hence we can concentrate on studying $\arg \tilde{g_t}'(z_0)$.

The driving function of the Loewner chain maps to a process on the unit circle by:
\[\tilde{\beta_t} = \frac{\beta_t - g_t(z_0)}{\beta_t - \overline{g_t(z_0)}}\]
Defining a time change
\[ds = \frac{(\tilde{\beta_t}-1)^4}{\abs{g_t(z_0) - \overline{g_t(z_0)}}^2\tilde{\beta_t}^2}dt\]
it is shown in \cite{BEF} that we can write the time evolution of $\tilde{g_t}$ 
as $h_s = \tilde{g_t}(s)$ where $h_s$ satisfies the following equation:
\[\partial_s h_s(z) = \frac{2\tilde{\beta_t}h_s(z)(h_s(z)-1)}{(1-\tilde{\beta_t})(h_s(z) - \tilde{\beta_t})}\]

Now differentiating this with respect to $s$ at $z = z_0$, we get
\[\partial_s h'_s(z_0) = \frac{2h'_s(z_0)}{1 - \tilde{\beta_s}}\]
Hence 
\[\partial_s \log h'_s(z_0) = \frac{2}{1 - \tilde{\beta_s}}\]
From here two things follow. Firstly, as 
\[\CR(z_0,H_t) = \frac{1}{\abs{h'_s(z_0)}}\]
and 
\[\partial_s \log \abs{h'_s(z_0)} = 1\]
we can follow the evolution of the conformal radius very exactly: 
\begin{equation}\label{eqR}
\partial_s \log \CR(z_0,H_t) = -1
\end{equation}
Secondly, after writing $\beta_s = \exp(i\alpha_s)$, a small calculation gives that 
we can also follow the winding:
\begin{equation}\label{eq1}
\partial_s \arg h'_s(z_0) = \cot \frac{\alpha_s}{2}
\end{equation}

Hence, everything is at our hand as soon as we understand the transformed driving process $\alpha_s$.

\paragraph{The diffusion of the driving process}\mbox{}\\
Indeed, putting faith in \cite{BEF}, Ito's formula gives that $\alpha_s$ defined as above by $\beta_s = \exp(i\alpha_s)$ is a diffusion in $(0,2\pi)$ starting from $\alpha_0 = 2\arg g_t(z_0)$ and satisfying the following stochastic differential equation:
\[d\alpha_s = \sqrt{\kappa}dB_s + \frac{\kappa - 4}{2}\cot \frac{\alpha_s}{2}ds\]
where $B_s$ is a standard 1D Brownian Motion. This is well-defined \& omits a unique strong solution up to the first exit-time. 

As for $\kappa < 4$ the drift term is attractive towards the boundary, then comparing to Brownian motion, one can conclude that the exit time $\tau$ for the diffusion is almost surely finite. Moreover, looking at $\ref{eqR}$, we can put the hitting time in exact correspondence with the conformal radius. Indeed, we have
 \[\CR(z_0,H_t) = \CR(z_0,\HH)e^{-\tau}\]
Moreover, from \eqref{eq1} the claimed form for the winding also follows:
\[w(z_0) = \int_0^\tau \partial_s \arg h'_s(z_0) = \int_0^\tau \cot \frac{\alpha_s}{2}ds\]
\end{proof}

\end{document}